\documentclass[12pt]{article}
\usepackage{a4wide}
\usepackage{epsfig}
\usepackage{amsfonts}
\usepackage{amsthm}
\usepackage{amsmath}
\usepackage{amssymb}
\usepackage{enumerate}
\usepackage{wasysym}
\usepackage{booktabs}

\usepackage[usenames]{color}
\newcommand{\twvec}[1]{{\def\arraystretch{0.85}\Bigl(\hskip1pt%
        \begin{matrix}#1\end{matrix}\hskip1pt\Bigr)}}

\small\normalsize
\allowdisplaybreaks[1]

\newtheorem{theorem}{Theorem}[section]
\newtheorem{lemma}[theorem]{Lemma}

\newtheorem{proposition}[theorem]{Proposition}
\newtheorem{problem}{Problem}
\newtheorem{conjecture}{Conjecture}

\def\bar{\overline}
\def\hat{\widehat}
\def\mod{{\; \mathrm{mod} \;}}

\def\I{{\mathcal I}}
\def\O{{\mathcal O}}
\def\R{{\mathbb R}}
\def\N{{\mathbb N}}

\def\chif{{k}}
\def\eps{{\varepsilon}}

\def\Bigcup{\displaystyle\bigcup\limits}

\def\Gceiti{the project P202/12/G061 CE-ITI}

\usepackage{graphicx}
\usepackage{ifpdf}
\ifpdf
\fi

\begin{document}
\title{Extensions of Fractional Precolorings\\[1mm]
  show Discontinuous Behavior\,\thanks{\,Research for this paper was
    started during visits of JvdH, DK, MK and JV to LIAFA\@. The authors
    like to thank the members of LIAFA for their hospitality. Visits of MK
    and JV to LIAFA and visits of J-SS to Charles University were supported
    by the PHC Barrande 24444 XD.}}
  \date{}
  \author{Jan van den Heuvel\,\thanks{
    \,Department of Mathematics, London School of Economics, Houghton
    Street, London WC2A 2AE, UK\@. E-mail:
    \texttt{j.van-den-heuvel@lse.ac.uk}.
  }
  \and Daniel Kr\'al'\,\thanks{
    \,Mathematics Institute, DIMAP and Department of Computer Science,
    University of Warwick, Coventry CV4 7AL, UK\@. Previous affiliation:
    Computer Science Institute, Faculty of Mathematics and Physics, Charles
    University, Prague, Czech Republic, where the author was supported by
    \Gceiti\@. E-mail: \texttt{D.Kral@warwick.ac.uk}.
 }
 \and Martin Kupec\,\thanks{
    \,Computer Science Institute, Faculty of Mathematics and Physics,
    Charles University, Prague, Czech Republic. The author was supported by
    {\Gceiti} and by the student grant GAUK 60310. E-mail:
    \texttt{kupec@iuuk.mff.cuni.cz}.
  }
  \and Jean-S\'ebastien Sereni\,\thanks{
    \,CNRS (LORIA), Vand\oe{}uvre-l\`es-Nancy, France. E-mail:
    \texttt{sereni@kam.mff.cuni.cz}.
  } 
  \and Jan Volec\,\thanks{
    \,Mathematics Institute and DIMAP, University of Warwick, Coventry CV4
    7AL, UK\@. Previous affiliation: Computer Science Institute, Faculty of
    Mathematics and Physics, Charles University, Prague, Czech Republic.
    E-mail: \texttt{honza@ucw.cz}.
  }
  }

\maketitle

\begin{abstract}
  \noindent
  We study the following problem: given a real number $k$ and integer $d$,
  what is the smallest $\varepsilon$ such that any fractional
  $(k+\varepsilon)$-precoloring of vertices at pairwise distances at least
  $d$ of a fractionally $k$-colorable graph can be extended to a fractional
  $(k+\varepsilon)$-coloring of the whole graph? The exact values of
  $\varepsilon$ were known for $k\in\{2\}\cup [3,\infty)$ and any $d$. We
  determine the exact values of $\varepsilon$ for $k\in (2,3)$ if $d=4$,
  and $k\in [2.5,3)$ if $d=6$, and give upper bounds for $k\in(2,3)$ if
  $d=5,7$, and $k\in(2,2.5)$ if $d=6$. Surprisingly, $\varepsilon$ viewed
  as a function of $k$ is discontinuous for all those values of~$d$.
\end{abstract}


\section{Introduction and main results}

Graph coloring is one of the classical topics in graph theory. In this
paper, we seek conditions when a precoloring of some vertices in a graph
can be extended to a coloring of the entire graph. This line of research
was initiated by Thomassen~\cite{bib-thomassen} who asked for sufficient
conditions on extending precolorings of vertices in planar graphs. His
original question led to the following result of
Albertson~\cite{bib-albertson}.

\begin{theorem}[\cite{bib-albertson}]
  \label{thm-albertson}
  Let $G$ be an $r$-colorable graph and $W$ a subset of its vertex set such
  that the distance between any two vertices of $W$ is at least four. Then
  every $(r+1)$-coloring of $W$ can be extended to an $(r+1)$-coloring of
  $G$.
\end{theorem}

This result initiated a line of research
\cite{bib-a2,bib-a3,bib-a4,bib-a5,bib-a6,bib-a7} seeking conditions for the
existence of an extension of a precoloring of various types of
subgraphs.

It is natural to ask whether an analogue of Theorem~\ref{thm-albertson}
also holds for non-integer relaxations of colorings. For circular colorings
introduced in~\cite{bib-vince}, the extension problem was almost completely
solved by Albertson and West~\cite{bib-west} (see~\cite{bib-zhu1,bib-zhu2}
for background and results on circular colorings).

Another well-established relaxation of classical colorings is the notion of
fractional colorings, see~\cite{bib-book}, which we address in this paper.
A \emph{fractional $k$-coloring} of a graph $G$ is an assignment of
measurable subsets of the interval $[0,k)\subseteq\R$ to the vertices of
$G$ such that each vertex receives a subset of measure one and adjacent
vertices receive disjoint subsets. The \emph{fractional chromatic number
  of~$G$} is the infimum over all positive real numbers $k$ such that $G$
admits a fractional \mbox{$k$-coloring}. For finite graphs (which we
restrict our attention to), such $k$ exists, the infimum is in fact a
minimum, and its value is always rational. A \emph{fractional
  \mbox{$k$-precoloring}} is an assignment of measurable subsets of measure
one of the interval $[0,k)$ to some vertices of a graph.

In this paper, we study conditions under which a fractional precoloring can
be completed to a fractional coloring of the whole graph.

\begin{problem}
  Let $\eps > 0$ be a real, $k\ge2$ a rational and $d\ge3$ an integer.
  Given a fractionally $k$-colorable graph $G$ and a fractional
  $(k+\eps)$-precoloring of a subset of its vertex set at pairwise distance
  at least $d$, is it possible to extend the precoloring to a fractional
  $(k+\eps)$-coloring of the whole graph $G$?
\end{problem}

For a fixed rational $k\ge 2$ and a fixed integer $d\ge3$, let $g(k,d)$ be
the infimum over all non-negative reals satisfying the following: for any
$\eps\ge g(k,d)$ and any fractionally $k$-colorable graph $G$, an arbitrary
$(k+\eps)$-precoloring of vertices at pairwise distance at least~$d$ in~$G$
can be extended to a fractional $(k+\eps)$-coloring of~$G$. The next
proposition, which is proved in~\cite{bib-extenze}, implies that for any
$\eps < g(k,d)$ there exists a fractionally $k$-colorable graph~$G$ with a
fractional $(k+\eps)$-precoloring of some of its vertices at pairwise
distance at least $d$, such that there is no extension of the precoloring
to a fractional $(k+\eps)$-coloring of $G$.

\begin{proposition}[\cite{bib-extenze}]
  \label{prop-monotone}
  Let $G$ be a graph with fractional chromatic number $k$ and $W$ a subset
  of its vertex set. The set of all non-negative reals $\eps$ such that any
  fractional $(k+\eps)$-precoloring of $W$ can be extended to a fractional
  $(k+\eps)$-coloring of $G$ is a closed interval.
\end{proposition}

The only value of $d$ for which the values of $g(k,d)$ are known for all
$k\ge 2$ is $d=3$. In this case, $g(k,3)=1$ for all $k\in [2,\infty)$,
see~\cite{bib-extenze}. For $d\ge 4$, the values of $g(k,d)$ for
$k\in\{2\}\cup [3,\infty)$ were determined in~\cite{bib-extenze}.

\begin{theorem}[\cite{bib-extenze}]
  \label{thm-ext1}
  For every $k\in\{2\}\cup [3,\infty)$ and $d\ge 3$, we have:
  \begin{equation*}
    g(k,d)=\left\{\begin{array}{cl}
        \dfrac{\sqrt{(d'k-1)^2+4d'(k-1)}-(d'k-1)}{2d'}\,, &
        \text{if $d\equiv0\mod4$;}\\[3mm]
        \dfrac{k-1}{d'k}\,, &
        \text{if $d\equiv1\mod4$;}\\[3mm]
        \dfrac{\sqrt{(d'k)^2+4d'(k-1)}-d'k}{2d'}\,, &
        \text{if $d\equiv2\mod4$;}\\[3mm]
        \dfrac{k-1}{d'k+k-1}\,, & \text{otherwise,}\end{array}\right.
  \end{equation*}
  where $d'=\lfloor d/4\rfloor$. The formula also holds for
  $k\in[2,\infty)$ and $d=3$.
\end{theorem}

The main goal of this paper is to shed more light on values of $g(k,d)$ for
$k\in (2,3)$. We determine the values of $g(k,d)$ for $k\in (2,3)$ if
$d=4$, and for $k\in [2.5,3)$ if $d=6$ (see Figures~\ref{fig-d4}
and~\ref{fig-d6}).

\begin{figure}[t]
  \begin{center}
    \epsfxsize=100mm
    \epsfbox{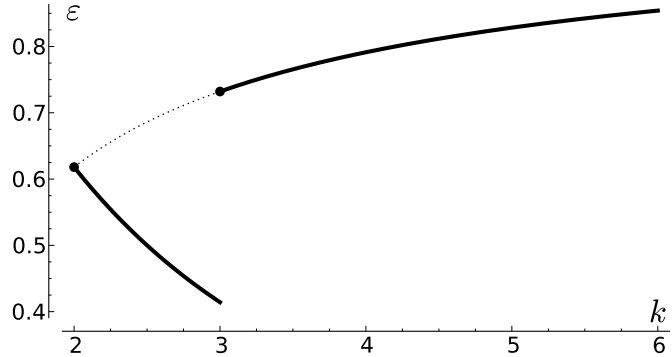}
    \caption{The values of $g(k,4)$. The dotted line represents the
      extension of $g(k,4)$ for $k\in\{2\}\cup[3,\infty)$ to $k\in(2,3)$.}
    \label{fig-d4}
  \end{center}
\end{figure}

\begin{theorem}
\label{thmain-d4}
  For $k\in[2,3)$ we have
  $g(k,4)=\frac12\bigl(\sqrt{(k-1)^2+4}-k+1\bigr)$.
\end{theorem}

\begin{theorem}
\label{thmain-d6k25}
  For $k\in\{2\}\cup[2.5,3)$ we have
  $g(k,6)=\frac12\bigl(\sqrt{k^2+4}-k\bigr)$.
\end{theorem}

For additional values of $k\in (2,3)$ and $d$, we provide upper bounds
(Theorems~\ref{thm-d0}, \ref{thm-d2}, \ref{thm-d1}, \ref{thm-d3},
and~\ref{thm-d7k25}) which we believe to be tight. See Figures~\ref{fig-d5}
and~\ref{fig-d7} for the bounds we can prove for $d=5$ and $d=7$. To our
surprise, for fixed $d \in \{4,5,6,7\}$, the function $g(k,d)$ is
discontinuous in $k$ at $k=3$, while for $d\in\{6,7\}$ the function
$g(k,d)$ is also discontinuous at $k=2.5$. We provide some additional
comments on those observations in Section~\ref{sect-open}. Also note that
the functions $g(k,4)$ and $g(k,6)$ are decreasing on the intervals $[2,3)$
and $[2.5,3)$, respectively, whereas for all $d\ge3$ the functions $g(k,d)$
are increasing on $k\in[3,\infty)$.

\begin{figure}[th]
  \begin{center}
    \epsfxsize=100mm
    \epsfbox{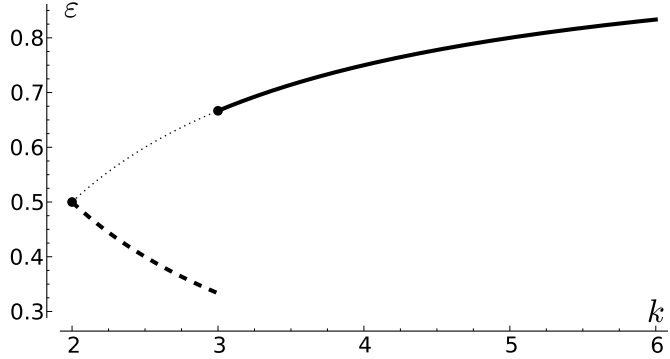}
    \caption{The values of $g(k,5)$. The dashed line gives the upper
        bound of $g(k,5)$ for $k\in(2,3)$. The dotted line represents the
      extension of $g(k,5)$ for $k\in\{2\}\cup[3,\infty)$ to $k\in(2,3)$.}
    \label{fig-d5}
  \end{center}
\end{figure}

\begin{figure}[th]
  \begin{center}
    \epsfxsize=100mm
    \epsfbox{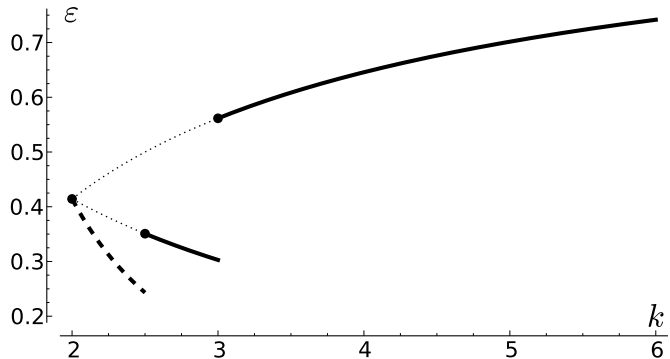}
    \caption{The values of $g(k,6)$. The dashed line gives the upper
        bound of $g(k,6)$ for $k\in(2,2.5)$. The dotted lines represent
      the extension of $g(k,6)$ for $k\in\{2\}\cup[3,\infty)$ to
      $k\in(2,3)$, and for $k\in\{2\}\cup[2.5,3)$ to $k\in (2,2.5)$,
      respectively.}
    \label{fig-d6}
  \end{center}
\end{figure}

\begin{figure}[th]
  \begin{center}
    \epsfxsize=100mm
    \epsfbox{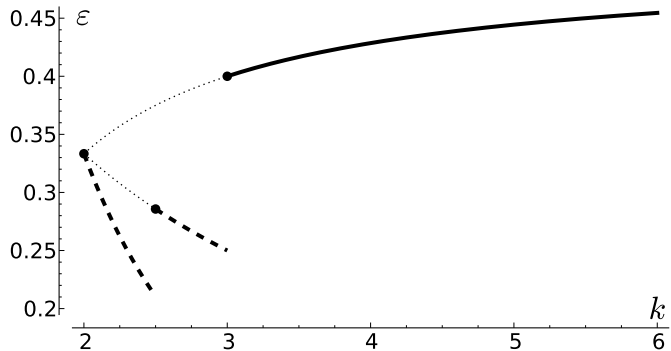}
    \caption{The values of $g(k,7)$. The dashed lines give the upper
        bound of $g(k,7)$ for $k\in(2,3)$. The dotted lines represent the
      extension of $g(k,7)$ for $k\in\{2\}\cup[3,\infty)$ to $k\in(2,3)$,
      and of the conjectured function for $k\in\{2\}\cup[2.5,3)$ to
      $k\in(2,2.5)$, respectively.}
    \label{fig-d7}
  \end{center}
\end{figure}

The paper is organized as follows. In the analysis of the values of
$g(k,d)$, we consider four cases based on the remainder of $d$ modulo 4. In
Section~\ref{sect-d0}, we present our upper bounds on $g(k,d)$ for
$k\in(2,3)$ and $d$ divisible by four. We also present the matching lower
bound for $d=4$. This lower bound is based on a simple expansion bound on
independent sets in Kneser graphs based on eigenvalues of its adjacency
matrix. In Section~\ref{sect-d2}, we present our upper bounds on $g(k,d)$
for $k \in (2,3)$ and $d$ congruent to two modulo four. This section also
contains the matching lower bound for the case $d=6$ and $k\in[2.5,3)$.
This lower bound uses a suitable solution of the linear program dual to
that for finding the fractional chromatic number of a Kneser graph.
Finally, in Sections~\ref{sect-upper-d1} and~\ref{sect-upper-d3} we present
our upper bounds on $g(k,d)$ for $d$ congruent to one and three,
respectively.


\section{Notation, definitions and preliminary results}

Before we can present our results, and their proofs, in detail, we need to
introduce some notation. For a positive integer $n$, we set
$[n]:=\{1,\ldots,n\}$. Next, for a set $Y \subseteq [0,\infty)$ we write
$2^Y$ for the set of all measurable subsets of $Y$. If $f:X\to 2^Y$ is a
mapping from a set~$X$ to $2^Y$ and $A$ is a subset of $X$, we write $f(A)$
for the set $\bigcup_{a\in A} f(a)$. We also write $g:X\hookrightarrow 2^Y$
for mappings from $X$ to $2^Y$ such that $g(i)\cap g(j) =\varnothing$ for
any two distinct $i,j\in X$.

We gave one possible definition of the fractional chromatic number of a
graph~$G$ in the introduction. An equivalent definition can be given as a
linear relaxation of determining the ordinary chromatic number: assign
non-negative real weights to the independent sets of $G$ such that for
every vertex $v\in V(G)$ the sum of the weights of independent sets
containing $v$ is at least one. The minimum possible sum of weights of all
independent sets in $G$, where the minimum is taken over all such
assignments, is equal to the fractional chromatic number of $G$.

The definition of fractional colorings also allows one to define a class of
\emph{universal} graphs, i.e., a class such that for every graph with
fractional chromatic number $k$ there is a homomorphism to one of the
graphs in this class. A \emph{homomorphism} from a graph $G$ to a graph~$H$
is a mapping $f:V(G)\to V(H)$ such that if $u$ and $v$ are two adjacent
vertices of $G$, then the vertices $f(u)$ and $f(v)$ are adjacent in $H$.
If such a mapping exists, we say that $G$ is \emph{homomorphic} to~$H$.

Universal graphs for fractional colorings are Kneser graphs $K_{p/q}$; the
graph $K_{p/q}$, for integers $1\le q\le p$, has a vertex set formed by all
$q$-element subsets of $[p]$, i.e.,
$V(K_{p/q})=\displaystyle\twvec{[p]\\q}$. Two vertices $A$ and $A'$ are
adjacent if $A\cap A'=\varnothing$. It is not hard to show that the
fractional chromatic number of $K_{p/q}$ is equal to $p/q$. The following
proposition can be found, e.g., in~\cite{bib-homo}.

\begin{proposition}
  \label{prop-fract}
  Let $G$ be a graph with fractional chromatic number $k$. There exist
  integers $p$ and $q$ such that $k=p/q$ and $G$ is homomorphic to the
  graph $K_{p/q}$.
\end{proposition}

Analogously to~\cite{bib-extenze}, our proofs are based on defining and
analyzing graphs that are universal for graphs (of a given fractional
chromatic number) with some precolored vertices. The graphs we introduce
now are isomorphic to the ones defined in~\cite{bib-extenze}, although we
use a slightly different notation.

The \emph{extension product} of two graphs $G$ and $H$ is the graph with
vertex set $V(G)\times V(H)$ such that vertices $(u,v)$ and $(u',v')$ are
adjacent if $u$ and $u'$ are adjacent in~$G$ and either $v=v'$, or $v$ and
$v'$ are adjacent in $H$. This type of a graph product was introduced by
Albertson and West~\cite{bib-west}. An equivalent notion was used
in~\cite{bib-extenze} under the name \emph{universal product}; the only
difference is that the meaning of $G$ and $H$ was swapped, i.e., the
universal product of~$G$ and $H$ is isomorphic to the extension product of
$H$ and $G$. For a set $X \in\twvec{[p]\\q}$, a \emph{ray} $R_{p,q,P}^X$ is
the extension product of the Kneser graph $K_{p/q}$ and the $(P+1)$-vertex
path with vertices $0,\ldots,P$; the vertex $(X,0)$ of $R_{p,q,P}^X$ is
marked as \emph{special}. The copy of $K_{p/q}$ in the ray $R_{p,q,P}^X$
corresponding to the vertex $P$ of the path is said to be the \emph{base}
of the ray. For brevity, $R_{p,q,P}$ will stand for $R_{p,q,P}^{[q]}$ in
what follows. The ray $R_{5,2,2}^{[2]}$ is sketched in Figure~\ref{fig-1}.
Note that the graph $R_{p,q,P}^X$ is homomorphic to $K_{p/q}$, and the
distance between the vertex $(X,0)$ and any vertex $(A,\ell)$, for
$A\in\twvec{[p]\\q}$ and $\ell\in[1,P]$, is at least $\ell$.
\begin{figure}
  \begin{center}
    \epsfbox{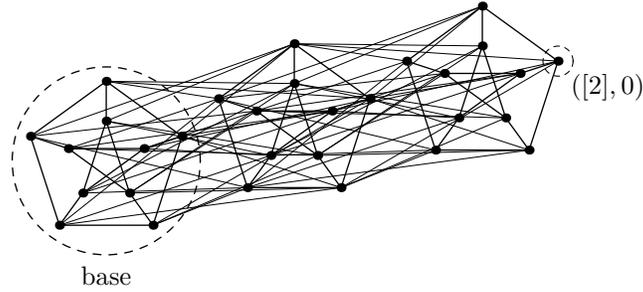}
  \end{center}
  \caption{The ray $R_{5,2,2}^{[2]}$.}
  \label{fig-1}
\end{figure}%

The graph $U_{p,q,d}^{n}$, which we now define, is a universal graph for
graphs with fractional chromatic number $p/q$ with $n$ precolored vertices
at pairwise distance at least $d$. Fix positive integers $p,q,d$ and $n$
such that $p\ge q/2$ and $d\ge3$. If $d$ is even, the graph $U_{p,q,d}^{n}$
is the extension product of the Kneser graph $K_{p/q}$ and the star
$K_{1,n{[p]\choose q}}$ with each edge subdivided $d/2-1$ times. For every
$X\in\twvec{[p]\\q}$, we mark the vertex $X$ as \emph{special} in $n$
copies of $K_{p/q}$ corresponding to the leaves of the star (for different
values of $X$, we choose different copies). In this way, the subgraphs of
$U_{p,q,d}^n$ corresponding to the products of the subdivided edges and
$K_{p/q}$ are isomorphic to rays $R_{p,q,d/2}^X$. Hence, the graph
$U_{p,q,d}^n$ can be viewed as obtained from $n$ copies of the ray
$R_{p,q,d/2}^X$ for each choice of $X\in\twvec{[p]\\q}$ through
identification of the bases of the rays. The graph $U^1_{5,2,6}$ is
sketched in Figure~\ref{fig-2}.
\begin{figure}
  \begin{center}
    \epsfbox{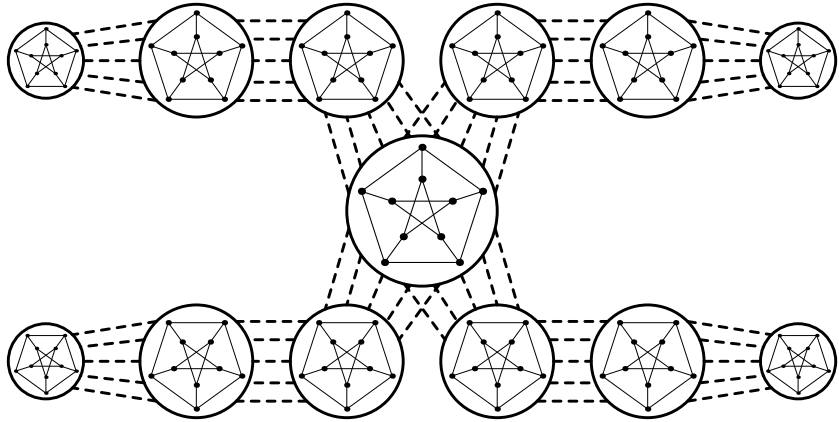}
  \end{center}
  \caption{A sketch of the graph $U^1_{5,2,6}$ (only 4 rays out of 10 are
    drawn).}
  \label{fig-2}
\end{figure}

For positive integers $N$ and $P$, let $L_{N,P}$ be the graph obtained from
a clique~$K_N$ by identifying each vertex of the clique with an end-vertex
of a $P$-vertex path; so $L_{N,P}$, for $P\ge2$, has $N\cdot(P-2)$ vertices
of degree two, $N$ vertices of degree one, and~$N$ vertices of degree $N$.
If $d$ is odd, the graph $U_{p,q,d}^n$ is the extension product of the
Kneser graph $K_{p/q}$ and the graph $L_{n{p\choose q},(d+1)/2}$. Again,
for each $X\in\twvec{[p]\\q}$, we mark vertices $X$ in $n$ of the copies of
$K_{p/q}$ corresponding to the vertices of degree one of
$L_{n{p\choose q},(d+1)/2}$ as special (with different copies for different
values of~$X$ again). In this way, we can view $U_{p,q,d}^n$ as a union of
$n\twvec{p\\q}$ rays $R_{p,q,(d-1)/2}^X$ with additional edges between
their bases. The graph $U^1_{5,2,7}$ is sketched in Figure~\ref{fig-3}.
\begin{figure}
  \begin{center}
    \epsfbox{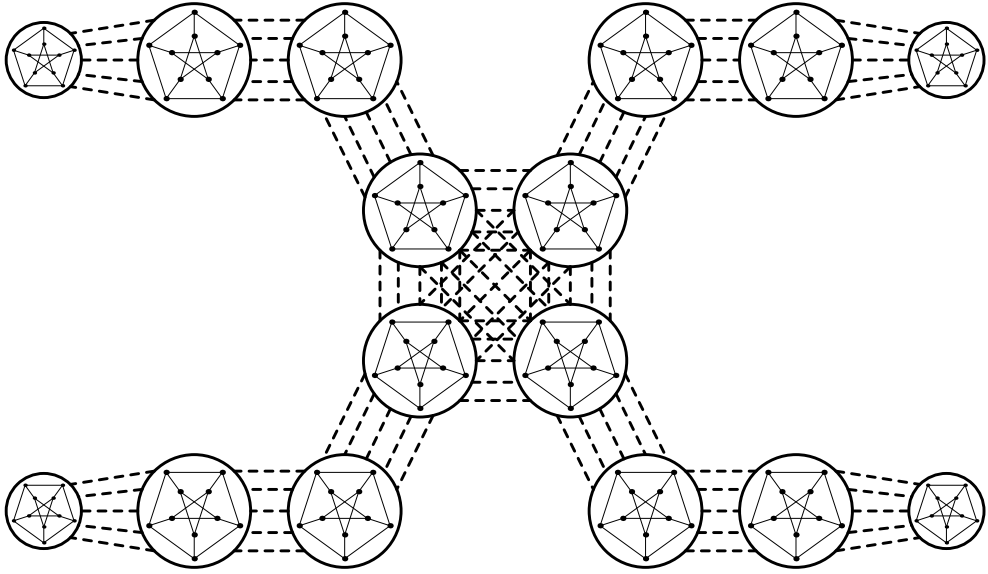}
  \end{center}
  \caption{A sketch of the graph $U^1_{5,2,7}$ (only 4 rays out of 10 are
    drawn).}
  \label{fig-3}
\end{figure}

In the next three propositions, we summarize the properties of the graphs
$U_{p,q,d}^n$ needed in the proofs. We start with the first two of them;
the proof of the first one is straightforward and the proof of the second
one is in~\cite{bib-extenze}.

\begin{proposition}
  \label{prop-universal}
  The graph $U_{p,q,d}^n$ for $p/q\ge2$ and $d\ge3$ is homomorphic to
  $K_{p/q}$ and its special vertices are at pairwise distance at least $d$.
\end{proposition}

\begin{proposition}[\cite{bib-extenze}]
  \label{prop-homo}
  Let $G$ be a graph with fractional chromatic number $k$ and~$W$ a subset
  of its vertex set at pairwise distance at least $d\ge 3$. There exist
  positive integers $p$ and~$q$, such that $k=p/q$ and the graph $G$ has a
  homomorphism to~$U_{p,q,d}^{|W|}$ that maps the vertices of $W$ to
  distinct special vertices of~$U_{p,q,d}^{|W|}$.
\end{proposition}

The length of the shortest odd cycle of a graph $G$ is the \emph{odd girth}
of~$G$. The odd girth of the Kneser graph $K_{p/q}$ is equal to
$2\Bigl\lceil\dfrac{q}{p-2q}\Bigr\rceil + 1$, see~\cite{bib-oddgirth}. Note
that Proposition~\ref{prop-fract} implies that if $G$ is a fractionally
$k$-colorable graph, then its odd girth is at least
$2\lceil1/(k-2)\rceil+1$. The main difference between the case
$\chif\in\{2\} \cup [3,\infty)$, which was fully analyzed
in~\cite{bib-extenze}, and the case $\chif \in (2,3)$ is that vertices of a
ray $R_{p,q,P}$ at some fixed small distance from the special vertex form
an independent set. Observe that the minimum distance for which this
property does not hold is related to the odd girth of the Kneser graph
$K_{p/q}$.

\begin{proposition}
  \label{prop-indep}
  Consider a special vertex $s$ of a universal graph $U^n_{p,q,d}$ and an
  integer
  $\ell\in\bigl\{1,2,\ldots,\Bigr\lceil\dfrac{q}{p-2q}\Bigr\rceil-1\bigr\}$.
  The vertices at distance $\ell$ from $s$ form an independent set in
  $U^n_{p,q,d}$.
\end{proposition}

Finally, we state the following proposition which is implicit in the proof
of Theorem~\ref{thm-ext1} in~\cite{bib-extenze}.

\begin{proposition}[\cite{bib-extenze}]
  \label{prop-pseudorandom}
  Let $k=p/q$ be rational, where $p,q \in \N$ and $p\ge2q$, $d,n \in \N$
  and $\eps > 0$. For every fractional $(k+\eps)$-precoloring of the
  special vertices of $U^n_{p,q,d}$ by subsets
  $C_1, C_2, \dots, C_{n{p\choose q}} \subseteq [0,k+\eps)$ there exist
  functions $f_o$ and $f_e$ from $[p]$ to $2^{[0,k+\eps)}$ such that the
  following holds:
  \vspace{-1mm}
  \begin{enumerate}[1)]
  \item for every {$i,j\in[p]$, $i\ne j$:} $f_o(i) \cap f_o(j)
    =\varnothing$ and $f_e(i) \cap f_e(j) =\varnothing$;
  \vspace{-1mm}
  \item for every $i \in [p]$ and $a \in \Bigl[n\twvec{p\\q}\Bigr]$:
    \vspace{-1mm}
    \begin{enumerate}[a)]
    \item $\|f_o(i)\|={(k+\eps)/p}$ and $\|f_o(i)\cap C_a \| =1/p$,
    \item $\|f_e(i)\|={1/q}$ and $\|f_e(i)\cap C_a\| =1/(p+q\eps)$.
    \end{enumerate}
  \end{enumerate}
\end{proposition}

\medskip
In other words, the function $f_o$ in Proposition~\ref{prop-pseudorandom}
is an equipartition of the interval $[0,k+\eps)$ into $p$ measurable parts
$f_o(1), \dots, f_o(p)$ such that the measure of the intersection of
$f_o(i)$ with each set $C_j$, for $i\in[p]$ and
$j \in \Bigl[n\twvec{p\\q}\Bigr]$, is the same as the expected intersection
of $C_j$ with a~random subset of $[0,k+\eps)$ of measure $(k+\eps)/p$.
Analogously, $f_e$ is a~partition of an appropriate subset of $[0,k+\eps)$
of measure $k$ into~$p$ measurable parts $f_e(1), \dots, f_e(p)$, where the
parts have measure $1/q$ and the measure of the intersection of $f_e(i)$
with each set $C_j$ is the same as for a random subset of $[0,k+\eps)$ of
measure $1/q$.


\section{Distances divisible by four}
\label{sect-d0}


\subsection{Upper bounds}
\label{sect-upper-d0}

In this section we prove upper bounds on $g(k,d)$ for $d\equiv0\mod 4$ in
the case that~$k$ and~$d$ satisfy $2\le k < 2 + \dfrac2{d-2}\,$. Observe
that Proposition~\ref{prop-indep} guarantees that if we consider the ray
$R_{p,q,d/2}$, then for any {$\ell\in\{1,\ldots,(d-2)/2\}$,} the vertices
at distance $\ell$ from the special vertex form an independent set.

\begin{lemma}
  \label{lemma-d0}
  Let $\eps$ be a positive real and $n$, $p$, $q$ and $d$ positive integers
  such that $d\equiv0\mod4$ and $p/q\ge2$. If the conditions
  \begin{gather}
    \label{ineq-d0-og}
     2\le k < 2+\frac{1}{2d'-1} \quad\text{and}\\
    \label{ineq-d0-bound}
     \eps \sum_{j=0}^{d'-2} (k-1)^{2j+2} + \eps \cdot
    \frac{k-1+\eps}{k+\eps} \ge \frac{1}{k+\eps}
  \end{gather}
  are satisfied, where $d'=d/4$ and $k=p/q$, then any fractional
  $(k+\eps)$-precoloring of the special vertices of $U_{p,q,d}^n$ can be
  extended to a fractional $(k+\eps)$-coloring of $U_{p,q,d}^n$.
\end{lemma}

\begin{proof}
  First observe that by Proposition~\ref{prop-monotone} we only need to
  consider the case that $\eps$ is the smallest positive real that
  satisfies inequality~\eqref{ineq-d0-bound}, i.e., that solves the
  equation
  \[\eps \sum_{j=0}^{d'-2} (k-1)^{2j+2} +
  \eps \cdot \frac{k-1+\eps}{k+\eps} = \frac{1}{k+\eps} \,.\]
  Furthermore, it is straightforward to show that any positive solution to
  this equation satisfies the following two inequalities as well:
  \begin{equation}
    \label{ineq-d0-aux1}
    \eps \sum_{j=0}^{d'-2} (k-1)^{2j} \le \frac{1}{k+\eps}
    \quad\text{and}\quad
    \eps \sum_{j=0}^{d'-2} (k-1)^{2j+1} \le \frac{k-1+\eps}{k+\eps}
    \,.
  \end{equation}
  These two inequalities will guarantee the existence of functions~$h_z$
  and~$g_y$, respectively, which we define later in the proof. Also note
  that the right inequality of~(\ref{ineq-d0-aux1}) is an immediate
  consequence of the left one.

  Now consider the universal graph $U_{p,q,d}^n$. Let $C_i$, for
  $i\in\Bigl[n\twvec{p\\q}\Bigr]$, be a precoloring of the special vertices
  and let~$f_e$ be a mapping as described in
  Proposition~\ref{prop-pseudorandom}. In what follows, for each ray $R_i$,
  which is isomorphic to $R_{p,q,2d'}$, we find a fractional coloring $c_i$
  that satisfies the following: for every set $A \in\twvec{[p]\\q}$, each
  vertex $v=(A,2d')$ of the base of $R_i$ is colored by the set $f_e(A)$,
  and the special vertex of $R_i$ is colored by $C_i$. Since the universal
  graph $U_{p,q,d}^n$ is constructed by identifying the vertices $(A,2d')$,
  the conclusion of the lemma follows from the existence of such a
  fractional coloring for each ray.

  Fix a ray $R_i$ and let $s$ be the special vertex of $R_i$. For an
  integer $\ell\in[2d'-1]$, let $V_\ell$ be the set of vertices of $R_i$ at
  distance $\ell$ from $s$, and let $V_{2d'}$ be the set of vertices of
  $R_i$ at distance at least $2d'$ from $s$. Observe that the sets
  $V_\ell$, $\ell=1,\ldots,2d'$, form a partition of
  $V(R_i)\setminus\{s\}$, and if a vertex $v=(A,\ell')$ of the ray~$R_i$ is
  in $V_\ell$, then $\ell' \le\ell$. In particular, the vertices of the
  base of $R_i$ form a subset of~$V_{2d'}$. By~\eqref{ineq-d0-og} and
  Proposition~\ref{prop-indep}, it follows that the set $V_\ell$ forms an
  independent set in $R_i$, for $\ell \in [2d'-1]$.

  The basic idea is to partition for each $V_\ell$ the interval
  $[0,k+\eps)$ into three parts. The first part will be split into $p$
  equal-size parts and will be assigned to vertices in $V_\ell$ according
  to the corresponding sets in the Kneser graph. The second part will be
  assigned to all vertices in $V_\ell$ (that is possible since $V_\ell$
  forms an independent set). The third part will not be used on the
  vertices of $V_\ell$ at all and will be reserved for the vertices in
  $V_{\ell-1}$. Based on the parity of $V_\ell$, either the second part
  will be inside $C_i$ and the third part will be disjoint from $C_i$, or
  vice versa. First we define the partition for $V_{2d'}$, and after
  defining the partition for some $V_{\ell}$, we define the partition for
  $V_{\ell-1}$. During this procedure, the sizes of the second and third
  parts will increase at the expense of the first part.

  Formally, we construct functions
  $f_x:[p] \hookrightarrow 2^{[0,k+\eps)}$,
  $g_y:[p] \hookrightarrow 2^{[0,k+\eps)}$ and
  $h_z:[p] \hookrightarrow 2^{[0,k+\eps)}$, for $x \in [2d']$,
  $y\in [d'-1]$ and $z\in [d'-1]$ in the following way. For $a \in [p]$ and
  $j = d'-1,d'-2,\dots,1$, we sequentially define:
  \begin{itemize}
  \item $g_{j}(a)$ as an arbitrary subset of $(f_e(a)\setminus C_i)
    \setminus \bigcup\limits_{j'=j+1}^{d'-1}g_{j'}(a)$\\[-2mm]
    \hspace*{\fill} of measure $\dfrac{\eps k}{p} (k-1)^{2(d'-j)-1}$,
    \vspace{-2mm}
  \item $h_{j}(a)$ as an arbitrary subset of $(f_e(a)\cap C_i) \setminus
    \bigcup\limits_{j'=j+1}^{d'-1}h_{j'}(a)$\\[-2mm]
    \hspace*{\fill} of measure $\dfrac{\eps k}{p} (k-1)^{2(d'-j)-2}$,
  \end{itemize}
  \vspace{-1mm}
  and then:
  \vspace{-1mm}
  \begin{itemize}
  \item $f_{2d'}(a):=f_e(a)$,
  \vspace{-1mm}
  \item $f_{2j+1}(a):=f_{2j+2}(a)\setminus h_j(a)$, and
  \vspace{-1mm}
  \item $f_{2j}(a):=f_{2j+1}(a)\setminus g_j(a)$.
  \end{itemize}
  Finally, we set $f_1(a):=f_2(a)\setminus C_i$ for every $a\in [p]$. Since
  the measure of~$f_e(a)$ is $1/q$ and the measure of $f_e(a) \cap C_i$ is
  $1/(p+q\eps)$, these functions exist if and only if
  the conditions~\eqref{ineq-d0-aux1} are satisfied. Next, we set $Y$ to be the
  set of measure $\eps$ that is disjoint from $f_e([p])$, i.e.,
  $Y:=[0,k+\eps)\setminus f_e([p])$. Observe that
  $\|Y\setminus C_i\|=\eps - \eps / (k+\eps)$. The described construction
  of the functions is sketched in Figure~\ref{fig-f-d0}.

  \begin{figure}
    \begin{center}
      \epsfxsize=135mm
      \epsfbox{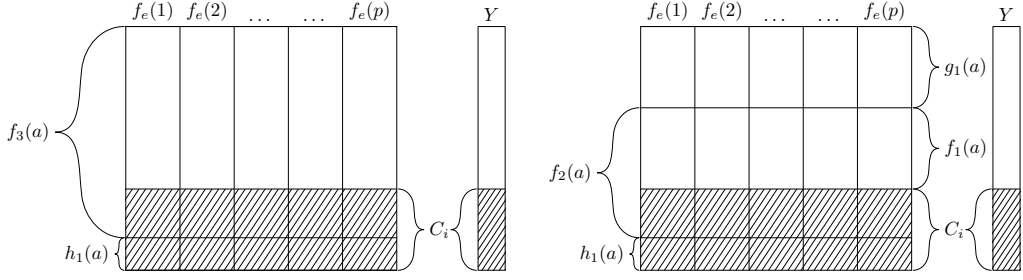}
    \end{center}
    \caption{The construction of a fractional coloring in
      Lemma~\ref{lemma-d0} for $d=8$.}
    \label{fig-f-d0}
  \end{figure}

  Let $\ell\in [2d']$ and $v=(A,\ell') \in V_\ell$. Recall that
  $\ell'\le\ell$. If $\ell$ is even, we set
  \begin{equation*}
    c_i(v):=f_\ell(A) \cup \bigcup_{j=\ell/2}^{d'-1} h_j([p])\,;
  \end{equation*}
  if $\ell\ge3$ is odd, we set
  \begin{equation*}
    c_i(v):=f_\ell(A) \cup \bigcup_{j={(\ell+1)/2}}^{d'-1} g_j([p]) \cup
    Y\,;
  \end{equation*}
  and for $\ell=1$ we set
  \begin{equation*}
    c_i(v):=f_1(A) \cup \bigcup_{j=1}^{d'-1} g_j([p]) \cup (Y\setminus
    C_i)\,.
  \end{equation*}
  Finally, we set $c_i(s):=C_i$.

  We claim that $\|c_i(v)\|\ge1$ for every vertex $v\in V(R_i)$. Indeed, if
  $v=s$, then the assertion immediately follows from $\|C_i\|=1$. Hence, in
  the remainder we may assume that~$v$ belongs to a set $V_\ell$ for some
  $\ell\in [2d']$. Observe that for a fixed $\ell\in [2d']$, the color sets
  of any two vertices $u$ and $v$ from $V_\ell$ have the same measure. Let
  $m_\ell$ be the measure of vertices in $V_\ell$. Then $m_{2d'}=1$, by the
  definition of $f_e$. If $d>4$, then $m_{2d'-1}=m_{2d'}$, since both $Y$
  and $h_{d'-1}(A)$, for $A \in {[p]\choose q}$, have measure $\eps$. Next,
  if $\ell \in \{3,5,\dots,2d'-3\}$, then
  \begin{equation*}
    m_{\ell} = m_{\ell+2} - \eps (k-1)^{2(d'-\lfloor \ell/2 \rfloor)-2} +
    (k-1) \cdot \eps (k-1)^{2(d'-\lceil \ell/2 \rceil)-1} =
    m_{\ell+2}\,.
  \end{equation*}
  Analogously, if $\ell \in \{2,4,\dots,2d'-2\}$, then
  \begin{equation*}
    m_{\ell} = m_{\ell+2} - \eps (k-1)^{2(d'-\ell/2)-1} + (k-1) \cdot \eps
    (k-1)^{2(d'-\ell/2)-2} = m_{\ell+2}\,.
  \end{equation*}
  Finally, for $m_1$ we have
  \begin{equation*}
    m_1={}1-\frac{1}{k+\eps} +
    (k-1) \cdot \eps \sum_{j=1}^{d'-1}(k-1)^{2j-1} + \eps -
    \frac{\eps}{k+\eps}\,,
  \end{equation*}
  which is at least one by \eqref{ineq-d0-bound}.

  It remains to check that the mapping $c_i$ assigns disjoint sets to any
  two adjacent vertices in $R_i$. Let $u=(A,\ell_u) \in V_{\ell_u}$ and
  $v=(B,\ell_v) \in V_{\ell_v}$ be two arbitrary adjacent vertices
  in~$R_i$. Hence, $A$ is disjoint from $B$ and without loss of generality
  $\ell_u \le \ell_v \le \ell_u + 1$. If $\ell_v = \ell_u$, then
  $\ell_v = 2d'$ (since for $\ell < 2d'$ the set $V_\ell$ is independent).
  Thus the sets $c_i(u)$ and $c_i(v)$ are disjoint, since $f_e(A)$ and
  $f_e(B)$ are disjoint.

  From now on, we assume that $\ell_v = \ell_u + 1$. If $\ell_v$ is even,
  then $c_i(v)$ is disjoint from~$Y$, and disjoint from $g_j([p])$ for any
  $j\in\{(\ell_u+1)/2,\ldots,d'-1\}$. Furthermore, $c_i(u)$ is disjoint
  from $h_j([p])$ for any $j\in\{\ell_v/2,\ldots,d'-1\}$. Analogously
  if~$\ell_v$ is odd and larger than one, then $c_i(v)$ is disjoint from
  $h_j([p])$ for any $j\in\{\ell_u/2,\ldots,d'-1\}$, and $c_i(u)$ is
  disjoint from~$Y$, and disjoint from $g_j([p])$ for any
  $j\in\{(\ell_v-1)/2,\ldots,d'-1\}$. Since $f_\ell(A)$ is a subset of
  $f_e(A)$ for any $\ell \in [2d']$ and $A\in\twvec{[p]\\q}$, the sets
  $c_i(u)$ and $c_i(v)$ are disjoint. Finally, the sets assigned to
  neighbors of $s$ are disjoint from $C_i$.

  We can conclude that the coloring $c_i$ is a fractional coloring of the
  ray $R_i$ with the required properties.
\end{proof}

Combining Lemma~\ref{lemma-d0} with Proposition~\ref{prop-homo} yields the
following theorem.

\begin{theorem}
  \label{thm-d0}
  Let $d$ be a positive integer such that $d\equiv0\mod4$, $k$ a rational
  and $\eps$ a positive real such that conditions \eqref{ineq-d0-og} and
  \eqref{ineq-d0-bound} are satisfied, where $d'=\lfloor d/4\rfloor$. If
  $G$ is a fractionally $k$-colorable graph and $W$ is a subset of its
  vertex set with pairwise distance at least $d$, then any fractional
  $(k+\eps)$-precoloring of $W$ can be extended to a fractional
  $(k+\eps)$-coloring of $G$.
\end{theorem}

\begin{proof}
  Let $p$ and $q$ be integers such that $k=p/q$, and $h$ the homomorphism
  from~$G$ to~$U_{p,q,d}^{|W|}$ given by Proposition~\ref{prop-homo}.
  Precolor the vertices of $h(W)$ with the colors assigned to their
  preimages. Note that this is possible since $h$ restricted to $W$ is
  injective. Since the parameters $k$, $\eps$ and $d$ satisfy the
  conditions \eqref{ineq-d0-og} and \eqref{ineq-d0-bound},
  Lemma~\ref{lemma-d0} yields that there exists an extension of this
  precoloring of $U^n_{p,q,d}$ to a fractional $(k+\eps)$-coloring $c_U$
  of~$U_{p,q,d}^{|W|}$. Since $h$ is a homomorphism of $G$ to
  $U^n_{p,q,d}$, setting $c(v):=c_U(h(v))$ for all $v\in V(G)$ yields a
  fractional $(k+\eps)$-coloring of~$G$ that extends the given precoloring
  of $W$.
\end{proof}


\subsection{Lower bound for distance four}
\label{sect-lower-four}

We start this section with the following proposition about the size of the
neighborhood of an independent set in a Kneser graph.

\begin{proposition}
  \label{prop-exp}
  Let $p$ and $q$ be positive integers, $p/q\ge2$. If $I$ is an independent
  set of the Kneser graph $K_{p/q}$, then
  $|N(I)| \ge \dfrac{p-q}{q}\cdot|I|\,$.
\end{proposition}

\begin{proof}
  Let $n=\twvec{p\\q}$ and $A=A(K_{p/q})$ be the normalized adjacency
  matrix of the Kneser graph $K_{p/q}$. This is the $n\times n$ matrix
  indexed by vertices of $K_{p/q}$ such that if $\{u,v\}$ is an edge of
  $K_{p/q}$, the entry corresponding to $(u,v)$ is equal to the inverse of
  the degree of~$u$, i.e., equal to $\twvec{p-q\\q}^{-1}$, while all other
  entries are zero. If $\lambda_1,\lambda_2,\dots,\lambda_n$ are the
  eigenvalues of $A$ such that
  $|\lambda_1| \ge |\lambda_2| \ge \dots \ge |\lambda_n|$, then it follows
  that $|\lambda_2| = \dfrac{q}{p-q}\,$, see~\cite{bib-spektrum}.

  A standard expansion inequality (see, e.g., \cite[Theorem
  4.15]{bib-wigderson}) asserts that
  \begin{equation}\label{eq-exp}
    |N_G(I)| \ge \frac{|I|}{(1-c)(\lambda_2)^2 + c}\,,
  \end{equation}
  for every vertex subset $I$ of a graph $G$ of size at~most $cn$, where
  $n$ is the number of vertices of $G$. If $I$ is an independent set of the
  Kneser graph $K_{p/q}$, then by the Erd\H{o}s-Ko-Rado Theorem (see, e.g.,
  \cite{bib-jukna}), the size of $I$ is at most $\twvec{p-1\\q-1}$.
  Therefore, $|I|/n \le q/p$, where $n=\twvec{p\\q}$, and hence $|N(I)| \ge
  \dfrac{p-q}{q}\cdot|I|$ by~\eqref{eq-exp}.
\end{proof}

Proposition~\ref{prop-exp} has a key role in proving that in any fractional
$k$-coloring of $K_{p/q}$, where $k\ge p/q$, there is a vertex $v$ such
that the union of sets assigned to the neighborhood of~$v$ has measure at
least $p/q-1$. Note that this statement is trivial if $p/q\ge3$, because in
that case the neighborhood of any vertex of $K_{p/q}$ is isomorphic to
$K_{p/q-1}$.

\begin{lemma}
  \label{lemma-nb}
  For every real $\eps\ge0$, all positive integers $p$ and $q$, where
  $p/q\ge2$, and any fractional $(p/q+\eps)$-coloring
  $c:V(K_{p/q}) \to 2^{[0,p/q+\eps)}$ of $K_{p/q}$, there exists a vertex
  $v\in V(K_{p/q})$ such that $\|c(N(v))\| \ge p/q - 1$.
\end{lemma}

\begin{proof}
  For $x\in[0,p/q+\eps)$, let $V_x \subseteq V(K_{p/q})$ be the set of
  vertices of $K_{p/q}$ that contain~$x$ in their color set, i.e.,
  $V_x = \{\,v\in V: x \in c(v)\,\}$. For $i \ge 1$ we define
  $X_i := \{\,x \in [0,p/q+\eps): |V_x|=i\,\}$. In other words, $X_i$ are
  the points in $[0,k+\eps)$ contained in exactly~$i$ color sets~$c(v)$.
  Note that for $i>\twvec{p-1\\q-1}$ the set $X_i$ is empty and that
  $\sum\limits_{i\ge1} i\cdot \|X_i\|=\twvec{p\\q}$. Next, let $X^j$ be the
  set of points $x \in [0,k+\eps)$ such that the number of vertices $v$
  that have at least one neighbor $u$ with $x \in c(u)$ is equal to $j$. In
  other words, let $X^j := \{\,x \in [0,p/q+\eps): |N(V_x)|=j\,\}$.

  Finally, consider all the intersections of $X_i$ with $X^j$, where
  $i \in \Bigl[\twvec{p-1\\q-1}\Bigr]$ and $j\in \Bigl[\twvec{p\\q}\Bigr]$,
  and let $X_i^j := \{\,x \in [0,p/q+\eps): |V_x|=i\;\text{and}\;
  |N(V_x)|=j\,\}$. Note that for a fixed~$j$ the sets $X_i^j$ form a
  partition of the set $X^j$, where for some values of $i$ the part $X_i^j$
  might be empty. Since for any $x$ the set $V_x$ forms an independent set
  in $K_{p/q}$, Proposition~\ref{prop-exp} yields that if $j<\frac{p-q}{q}
  \cdot i$, then $X_i^j$ is empty. Now, for a vertex $v\in V$, consider the
  measure of points $x \in [0,p/q+\eps)$ such that $x$ is contained in the
  color set of at least one neighbor of $v$. By a double counting argument
  it follows that
  \begin{equation*}
    \sum_{v\in V}{\|c(N(v))\|} =
    \sum_{j=1}^{p\choose q\vphantom{1}} j \cdot \|X^j\| =
    \sum_{j=1}^{p\choose q\vphantom{1}} \sum_{i=1}^{p-1\choose q-1}
    j \cdot \|X_i^j\| =
    \sum_{i=1}^{p-1\choose q-1} \sum_{j=1}^{p\choose q\vphantom{1}}
    j \cdot \|X_i^j\| \,.
  \end{equation*}
  Since the sets $X_i^j$ are empty for $j<\frac{p-q}{q} \cdot i$,
  we conclude that
  \begin{equation*}
    \sum_{v\in V}{\|c(N(v))\|} \ge
    \sum_{i=1}^{p-1 \choose q-1}\frac{p-q}{q}
    \cdot i \cdot \|X_i\| = \frac{p-q}{q}\cdot\twvec{p\\q}\,.
  \end{equation*}
  Therefore, there exists a vertex $v \in V$ such that
  $\|c(N(v))\| \ge p/q-1$.
\end{proof}

We are now ready to prove that the upper bound on $g(k,4)$ for $k\in[2,3)$
given in Theorem~\ref{thm-d0} is best possible. The proof uses the same
precoloring as was used in~\cite{bib-extenze} for a lower bound in the case
$k\in \{2\}\cup[3,\infty)$, but the argument for $k\in (2,3)$ is
considerably more involved.

\begin{theorem}
  \label{thm-lb-d4}
  Let $k \in [2,3)$ be a rational and $\eps$ a positive real such
  that $\eps < \dfrac{1+\eps}{k+\eps}\,$. There exist a graph $G$
  with fractional chromatic number $k$, a subset $W$ of its vertex set at
  pairwise distance at least four and a fractional $(k+\eps)$-precoloring
  of $W$ that cannot be extended to a fractional $(k+\eps)$-coloring of
  $G$.
\end{theorem}

\begin{proof}
  Let $\eps_0$ be the positive root of the equation
  $x = \dfrac{1+x}{k+x}\,$, i.e., let
  \begin{equation*}
    \eps_0:=\frac{1-k+\sqrt{(k-1)^2+4}}{2}\,.
  \end{equation*}
  Next, let $p',q$ be positive integers such that
  $k+\eps \le p'/q < k+\eps_0$ and $kq$ an integer. Set $\eps' := p'/q-k$,
  $p := kq$ and $G := U_{p,q,4}^n$, where $n=\twvec{p'\\q}$. We will show
  the existence of a $(k+\eps')$-precoloring of the special vertices of $G$
  that cannot be extended to a $(k+\eps')$-coloring of $G$. This implies
  that there exists also a $(k+\eps)$-precoloring of the special vertices
  that cannot be extended to a $(k+\eps)$-coloring of $G$ by
  Proposition~\ref{prop-monotone}. Since the special vertices of~$G$ are at
  pairwise distance at least four, the statement of the theorem immediately
  follows.

  Let $f:[p']\hookrightarrow 2^{[0,k+\eps')}$ be the function
  $f(i)=[(i-1)/q,i/q)$, for $i\in[p']$. Consider a precoloring of $G$ that
  assigns to the $n$ special vertices of the copies of $R_{p,q,2}^Y$, where
  $Y \in\twvec{[p]\\q}$, all the $n$ different sets $f(X)$, for
  $X\in\twvec{[p']\\q}$. We claim that this fractional
  $(k+\eps')$-precoloring cannot be extended to a fractional coloring of
  the whole graph.

  Suppose for contradiction that there exists an extension of the
  precoloring given by $f$ to a fractional coloring
  $c:V(G)\to 2^{[0,k+\eps')}$. Let $H$ be the base of $G$ (that is, the
  common bases of all rays). Since $H$ is isomorphic to $K_{p/q}$,
  Lemma~\ref{lemma-nb} implies that there exists a vertex $v \in V(H)$ with
  $\|c(N_H(v))\| \ge k - 1$. Let $C:=c(N_H(v))$ and let $u$ be an arbitrary
  neighbor of $v$ in $H$; without loss of generality~$u$ is the vertex
  corresponding to the vertex $([q],2)$ in each ray of $U_{p,q,4}^n$.

  Now consider all the rays $S^{[q]}_{p,q,2}$ in $U_{p,q,4}$; by the
  definition of $f$, for any $X\in\twvec{[p']\\q}$ there is a ray where the
  special vertex $[q]$ is precolored with $f(X)$. Since each point of
  $[0,p'/q)$ is contained in exactly $\twvec{p'-1\\q-1}$ sets $f(X)$, a
  double counting argument yields that
  \begin{equation*}
    \|C\| = \frac1{\twvec{p'-1\\q-1}} \sum_{X\in{[p']\choose q}} \|C \cap
    f(X)\|\,.
  \end{equation*}
  Therefore, there exists $X\in\twvec{[p']\\q}$ such that
  $\|C \cap f(X)\| \le\dfrac{q}{p'} \|C\| $. Consider the corresponding ray
  $S$ with the special vertex $[q]$ precolored by $f(X)$, and let $v_1$ be
  the vertex $(v,1)$ in $S$. Observe that the neighborhood of $v_1$ in $G$
  contains $N_H(v) \cup \{s\}$, where $s$ is the special vertex of $S$.
  Therefore,
  \begin{equation*}
    \| c(N(v_1))\| \ge \|C\| + 1 - \|C \cap f(X) \| \ge k -
    \frac{q}{p'}(k-1)=k-\frac{k-1}{k+\eps'}\,.
  \end{equation*}
  Since $0 < \eps' < \eps_0$, it follows that $\eps' < (1+\eps')/(k+\eps')$
  and hence
  \begin{equation*}k+\eps' - \| c(N(v_1))\| < 1\,.
  \end{equation*}
  This implies that $c(v_1)$ intersects $c(N(v_1))$, a contradiction.
\end{proof}


\section{Distances congruent to two mod four}
\label{sect-d2}


\subsection{Upper bounds}
\label{sect-upper-d2}

We start this section with showing upper bounds for $g(k,d)$, for $d\equiv2
\mod 4$ such that $k$ and $d$ satisfy $k < 2 + \dfrac2{d-2}\,$. The
construction of the colorings for this choice of $k$ and $d$ is similar to
the one in~Lemma~\ref{lemma-d0}. However, since the parity of the length of
the rays in~$U^n_{p,q,d}$ is different, we need to swap the order in which
we define the functions $g_y$ and $h_z$ when we go from the base of a ray
to its special vertex.

\begin{lemma}
  \label{lemma-d2}
  Let $\eps$ be a positive real and $n$, $p$, $q$ and $d$ positive integers
  such that $d\ge6$, $d\equiv2 \mod 4$ and $p/q\ge2$. If the conditions
  \begin{gather}
    \label{ineq-d2-og}
    2\le k < 2+\frac{1}{2d'} \quad\text{and}\\
    \label{ineq-d2-bound}
    \eps \sum_{j=0}^{d'-1} (k-1)^{2j+1} \ge \frac{1}{k+\eps}
  \end{gather}
  are satisfied, where $d'=\lfloor d/4\rfloor$ and $k=p/q$, then any
  fractional $(k+\eps)$-precoloring of the special vertices of
  $U_{p,q,d}^n$ can be extended to a fractional $(k+\eps)$-coloring of
  $U_{p,q,d}^n$.
\end{lemma}

\begin{proof}
  As in the proof of Lemma~\ref{lemma-d0}, we only need to consider the
  case that $\eps$ is the positive solution of
  \begin{equation*}
    \eps \sum_{j=0}^{d'-1} (k-1)^{2j+1} = \frac{1}{k+\eps}\,.
  \end{equation*}
  Any such solution also satisfies the following two inequalities:
  \begin{equation}
    \label{ineq-d2-aux1}
    \eps \sum_{j=0}^{d'-2} (k-1)^{2j+1} \le \frac{1}{ k+\eps}
    \quad\text{and}\quad
    \eps \sum_{j=0}^{d'-1} (k-1)^{2j} \le \frac{k - 1 +\eps}{k+\eps}
    \,.
  \end{equation}

  For the universal graph $U_{p,q,d}^n$, let $C_i$, for $i\in
  \Bigl[n\twvec{p\\q}\Bigr]$, be a precoloring of the special vertices and
  let~$f_e$ be a mapping as described in
  Proposition~\ref{prop-pseudorandom}. Analogously to the proof of
  Lemma~\ref{lemma-d0}, for each ray $R_i$ we find a fractional coloring
  $c_i$ that satisfies the following: for every set $A \in\twvec{[p]\\q}$,
  each vertex $v=(A,2d'+1)$ of the base of $R_i$ is colored by the set
  $f_e(A)$, and the special vertex of $R_i$ is colored by $C_i$.

  Fix a ray $R_i$ and let $s$ be the special vertex of $R_i$. For an
  integer $\ell \in [2d']$, let $V_\ell\subseteq V(R_i)$ be the set of
  vertices at distance $\ell$ from $s$, and let $V_{2d'+1}$ be the set of
  vertices of $R_i$ at distance at least $2d'+1$ from $s$. Similarly to the
  proof of Lemma~\ref{lemma-d0}, the vertices of the base of $R_i$ form a
  subset of $V_{2d'+1}$, and the set $V_\ell$ forms an independent set in
  $R_i$, for $\ell\in[2d']$.

  We construct functions $f_x:[p] \hookrightarrow 2^{[0,k+\eps)}$,
  $g_y:[p] \hookrightarrow 2^{[0,k+\eps)}$ and
  $h_z:[p] \hookrightarrow 2^{[0,k+\eps)}$, for $x \in [2d'+1]$,
  $y\in [d']$ and $z\in [d'-1]$ as follows. For $a \in [p]$ and
  $j=d'-1,d'-2,\dots\,1$ we sequentially define:
  \vspace{-1mm}
  \begin{itemize}
  \item $h_{j}(a)$ as an arbitrary subset of $(f_o(a)\cap C_i) \setminus
    \bigcup\limits_{j'=j+1}^{d'-1}h_{j'}(a)$\\[-2mm]
    \hspace*{\fill} of measure $\dfrac{\eps k}{p}(k-1)^{2(d'-j)-1}$,
  \end{itemize}
  \vspace{-1mm}
  Next, we sequentially define for $a\in[p]$ and $m=d',d'-1,\dots,1$
  \vspace{-1mm}
  \begin{itemize}
  \item $g_{m}(a)$ as an arbitrary subset of $(f_e(a)\setminus C_i)
    \setminus \bigcup\limits_{m'=m+1}^{d'}g_{m'}(a)$\\[-2mm]
    \hspace*{\fill} of measure $\dfrac{\eps k}{p}(k-1)^{2(d'-m)}$,
  \end{itemize}
  \vspace{-1mm}
  and then:
  \vspace{-1mm}
  \begin{itemize}
  \item $f_{2d'+1}(a):=f_e(a)$,
  \vspace{-1mm}
  \item $f_{2m+1}(a):=f_{2m+2}(a)\setminus h_m(a)$ for $m < d'$, and
  \vspace{-1mm}
  \item $f_{2m}(a):=f_{2m+1}(a)\setminus g_m(a)$.
  \end{itemize}
  Finally, we set $f_1(a):=f_2(a)\setminus C_i$ and
  $Y:=[0,k+\eps)\setminus f_e([p])$. Similarly as in the proof of
  Lemma~\ref{lemma-d0}, such functions exist if and only if conditions
  \eqref{ineq-d2-aux1} are satisfied. The described construction of the
  functions is sketched in Figure~\ref{fig-f-d2}.

  \begin{figure}
    \begin{center}
      \epsfxsize=80mm
      \epsfbox{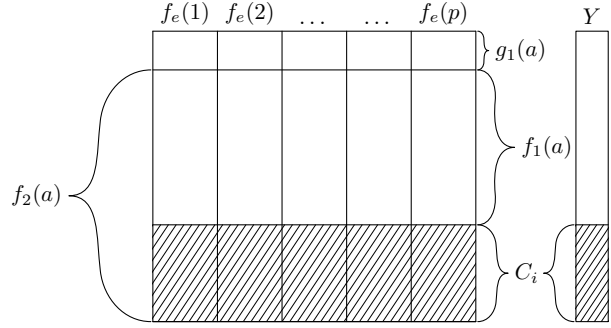}
    \end{center}
    \caption{The construction of a fractional coloring in
      Lemma~\ref{lemma-d2} for $d=6$.}
    \label{fig-f-d2}
  \end{figure}

  Let $\ell \in [2d']$ and $v=(A,\ell') \in V_\ell$. If $\ell$ is even, we
  set
  \begin{equation*}
    c_i(v):=f_\ell(A) \cup \bigcup_{j=\ell/2}^{d'-1} h_j([p]) \cup Y\,;
  \end{equation*}
  and if $\ell$ is odd, we set
  \begin{equation*}
    c_i(v):=f_\ell(A) \cup \Bigcup_{j={(\ell+1)/2}}^{d'-1} g_j([p])\,.
  \end{equation*}

  Setting $c_i(s):=C_i$, together with an analysis analogous to the that
  presented in the proof of Lemma~\ref{lemma-d0}, yield that $c_i$ is a
  fractional coloring of the ray $R_i$ with the required properties.
\end{proof}

Combining the lemma with Proposition~\ref{prop-homo} yields the following
theorem.

\begin{theorem}
  \label{thm-d2}
  Let $d$ be an integer such that $d\ge6$ and $d\equiv2\mod4$, $k$ a
  rational and~$\eps$ a positive real such that conditions
  \eqref{ineq-d2-og} and \eqref{ineq-d2-bound} are satisfied, where
  $d'=\lfloor d/4\rfloor$. If $G$ is a fractionally $k$-colorable graph and
  $W$ is a subset of its vertex set with pairwise distance at least $d$,
  then any fractional $(k+\eps)$-precoloring of $W$ can be extended to a
  fractional $(k+\eps)$-coloring of $G$.
\end{theorem}

We close this section by showing an upper bound on $g(k,6)$ for $k\in
[2.5,3)$, which is best possible due to Theorem~\ref{thm-lb-d6k25}. The
idea for the way we color the first neighborhood of each special vertex is
analogous to the one in Lemma~\ref{lemma-d2}. However, since the second
neighborhood of a special vertex does not form an independent set anymore,
we need to use a different strategy for coloring the second neighborhoods.

\begin{theorem}
  \label{thm-d6k25}
  Let $k$ be a positive rational less than 3 and $\eps$ a positive real
  such that $\eps \ge \dfrac{1}{k+\eps}\,$. If $G$ is a fractionally
  $k$-colorable graph and $W$ is a subset of its vertex set with pairwise
  distance at least six, then any fractional $(k+\eps)$-precoloring of $W$
  can be extended to a fractional $(k+\eps)$-coloring of $G$.
\end{theorem}

\begin{proof}
  By Proposition~\ref{prop-homo}, it is enough to consider only the
  universal graphs $U_{p,q,6}^n$, where $p/q=k$ and $n\in \N$, and an
  arbitrary precoloring of its special vertices. As in the proofs of
  Lemmas~\ref{lemma-d0} and \ref{lemma-d2}, let $C_i$, for
  $i\in \Bigl[n\twvec{p\\q}\Bigr]$, be a precoloring of the special
  vertices and let $f_e$ be a mapping as described in
  Proposition~\ref{prop-pseudorandom}. For each ray $R_i$ we find a
  fractional coloring~$c_i$ that satisfies the following: for every set
  $A \in\twvec{[p]\\q}$, each vertex $v=(A,3)$ of the base of $R_i$ is
  colored by the set $f_e(A)$, and the special vertex of $R_i$ is colored
  by $C_i$.

  Fix a ray $R_i$, let $s$ be the special vertex of $R_i$ and set
  $Y:=[0,k+\eps) \setminus f_e([p])$. By symmetry, it is enough to consider
  the case where $R_i$ is a copy of $R_{p,q,3}^{[q]}$. We construct
  functions $g:[q]\hookrightarrow 2^{[0,k+\eps)}$ and
  $h:[q]\hookrightarrow 2^{[0,k+\eps)}$ as follows. For $j \in [q]$ we
  define $g(j)$ to be an arbitrary subset of $f_e(j)\setminus C_i$ of
  measure $\dfrac{\eps}{p+q\eps}\,$. Note that these subsets always exist
  since $\|f_e(j)\setminus C_i\| = \dfrac{k-1+\eps}{p+q\eps}\,$. Next, we
  define $h(1),h(2),\dots,h(q)$ as an arbitrary equipartition of
  $Y \cap C_i$ into $q$ parts of measure $\dfrac{\eps}{p+q\eps}\,$. The
  described construction of the functions is sketched in
  Figure~\ref{fig-f-d6k25}.

  \begin{figure}
    \begin{center}
      \epsfxsize=115mm
      \epsfbox{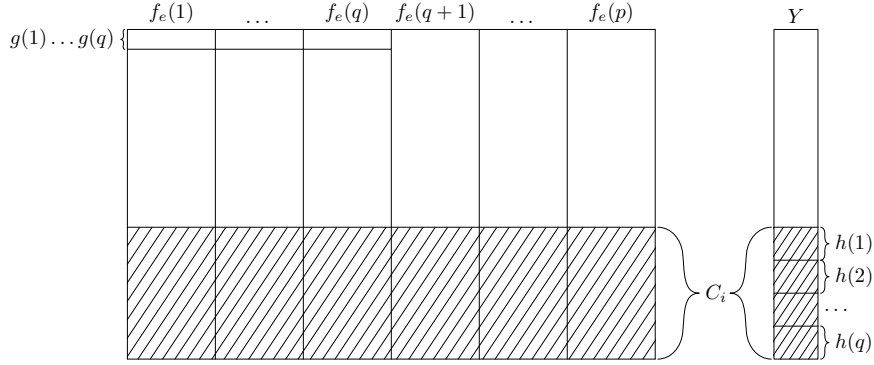}
    \end{center}
    \caption{The construction of a fractional coloring in
      Theorem~\ref{thm-d6k25}.}
    \label{fig-f-d6k25}
  \end{figure}

  Recall that the neighborhood of $s$ in $R_i$ forms an independent set.
  Since $s=([q],0)$, for every neighbor $(A,\ell')$ of $s$ we have
  $A\cap[q]=\varnothing$ and $\ell'=1$. We now construct a fractional
  coloring of $R_i$. Let $v=(A,\ell')$ be a vertex of $R_i$ and let $\ell$
  be the distance of $v$ from $s$ in~$R_i$. We define $c_i(v)$ in the
  following way:
  \vspace{-1mm}
  \begin{itemize}
  \item if $\ell\ge3$, then $c_i(v):=f_e(A)$;
  \vspace{-1mm}
  \item if $\ell=2$, then $c_i(v):=( f_e(A) \setminus g(A\cap[q])) \cup h(A
    \cap [q])$;
  \vspace{-1mm}
  \item if $\ell=1$, then $c_i(v):=( f_e(A) \setminus C_i ) \cup (Y
    \setminus C_i ) \cup g([q])$; and
  \vspace{-1mm}
  \item $c_i(s):=C_i$.
  \end{itemize}
  \vspace{-1mm}
  It is straightforward to check that we assigned disjoint sets to any two
  neighbors in $R_i$, and that any vertex at distance at least two from $s$
  is assigned a set of measure one. Furthermore, for every
  $A\in\twvec{[p]\setminus [q]\\q}$ the set $f_e(A)\setminus C_i$ is
  disjoint from both $Y$ and $g([q])$, and has measure
  $1-\dfrac{1}{k+\eps}\,$. Since $(Y \setminus C_i ) \cup g([q])$ has
  measure $\eps \ge\dfrac{1}{k+\eps}$ (recall that
  $\|Y \cap C_i\|=\eps/(k+\eps)$), it follows that~$c_i$ is a fractional
  coloring with the required properties.
\end{proof}


\subsection{Lower bound for distance six}
\label{sect-lower-d6k25}

The goal of this section is to prove that
$g(k,6)=\frac12\bigl(\sqrt{k^2+4}-k\bigr)$ for $k\in[2.5,3)$, i.e.,
$g(k,6)$ is the positive root of the equation $x=\dfrac1{k+x}\,$ for~$k$ in
that range. Before we present a formal proof, let us first sketch the idea.
Suppose for a contradiction that there exist $k\in[2.5,3)$ and
$\eps>0$ such that $g(k,6)\le\eps$ and $\eps<\dfrac1{k+\eps}\,$. As in the
proof of Theorem~\ref{thm-lb-d4}, we may assume that $\eps$ is a rational.
Let $p'$ and $q$ be integers such that $p'/q=k+\eps$ and $kq$ is an
integer, and let $p=kq$. We construct a precoloring of the special vertices
of $U_{p,q,6}^n$, where $n=\twvec{p'\\ q}$, such that each point of the
interval $[0,k+\eps)$ belongs to exactly
$\twvec{p\\q}\cdot\twvec{p'-1\\q-1}$ sets.

Now fix an arbitrary extension $c:V(U^n_{p,q,6})\to 2^{[0,k+\eps)}$ of this
precoloring to a fractional coloring of $U^n_{p,q,6}$, and let~$c'$ be the
restriction of~$c$ to the special vertices and the common bases of the
rays. For every ray $R_i$ of $U^n_{p,q,6}$, we will consider a linear
program $P_i$ that minimizes the value of a fractional coloring that
extends $c'$. Clearly, for every ray $R_i$ the optimal solution of $P_i$
has value at most $k+\eps$. On the other hand, we will show that there is a
ray $R_i$ such that the optimum of the dual program to $P_i$ is at least
$k+\dfrac1{k+\eps}\,$. Therefore, by weak duality of linear programming
(see e.g.~\cite{bib-schrijver}), it follows that $k + \dfrac1{k+\eps} \le k
+ \eps$, which contradicts the assumption.

We start the formal exposition by introducing the notion dual to fractional
colorings. Let $G=(V,E)$ be a graph. We say that a mapping $x:V(G)\to[0,1]$
is a \emph{fractional clique} in $G$ if for every independent set $I$ of
$G$ the sum $\sum\limits_{v\in I} x(v)$ is at most one. The \emph{weight
  of~$x$} is the sum $\sum\limits_{v\in V(G)} x(v)$. The problem of
determining the maximum weight of a fractional clique in $G$ can be also
formulated as a linear program. This program is the dual program to the
program that determines the fractional chromatic number of $G$. For a
fractional clique $x$ in a graph $G$ and a vertex subset $S \subseteq V(G)$
we set $x(S) := \sum\limits_{v\in S} x(v)$.

The following proposition asserts that for every $p$ and $q$, where $p/q
\in [2.5,3)$ and $q$ is even, there exist a maximum fractional clique 
and a vertex $v$ in the Kneser graph $K_{p/q}$ such that the sum of the weights
over $N(v)$ is equal to one.
\begin{proposition}
  \label{prop-fclique}
  For every positive integer $p$ and for every positive even integer $q$
  such that $p/q \in [2.5,3)$, there exists a fractional clique
  $x:\twvec{[p]\\q}\to [0,1]$ in $K_{p/q}$ of weight $p/q$ such that
  $x(v)=\twvec{p-q\\q}^{-1}$
  for every neighbor $v$ of the vertex $[q]$.
\end{proposition}

\begin{proof}
  Let $V_0:=\{[q]\}$, $V_1:= \twvec{[p]\setminus [q] \\ q}$ and
  $V_2:=\{\,X \in\twvec{[p]\\q}:|X\cap[q]|=q/2\,\}$ be vertex subsets of
  $K_{p/q}$. Note that $|V_1|=\twvec{p-q\\ q}$,
  $|V_2|=\twvec{q\\q/2} \cdot \twvec{p-q\\ q/2}$, and the sets $V_0$, $V_1$
  and $V_2$ are pairwise disjoint. Let $H$ be the subgraph of $K_{p/q}$
  induced by $V_0 \cup V_1 \cup V_2$. Observe that $H$ is connected since
  $p/q \ge 2.5$. We will show the existence of a fractional clique $x$ in
  $H$ of weight $k$ such that $x(v_1) ={|V_1|}^{-1}$ for each vertex $v_1
  \in V_1$. The statement of the proposition then follows.

  For each vertex $v \in V(H)$, define
  \begin{equation*}
    x(v) := \left\{\begin{array}{cl}
        3-p/q ,& \text{if $v=[q]$;} \\[1mm]
        \dfrac1{|V_1|}\,, & \text{if $v\in V_1$;}\\[3mm]
        \dfrac{2(p/q-2)}{|V_2|}\,, & \text{if $v\in V_2$.}
      \end{array}\right.
  \end{equation*}
  The weight of $x(H)$ is equal to $p/q$. Hence, it remains to check that
  $x(I) \le 1$ for every independent set $I$ of $H$.

  Fix an independent set $I$ of $H$. First suppose that $[q]\in I$. Observe
  that since $I$ has to be disjoint from $V_1$, it is enough to show that
  $|I \cap V_2| \le |V_2|/2$. Consider the subgraph $H'$ induced by $V_2$.
  Since every vertex in $H'$ has degree $\twvec{p-3q/2\\ q/2}$, the graph
  $H'$ is regular and therefore it has independence number at most
  $|V(H')|/2$.

  In the remainder of the proof we suppose that $[q] \notin I$. Next set
  $S_1:=I\cap V_1$, $S_2:=\{\,X \in \twvec{[p] \setminus [q]\\ q/2}:
  \text{there is a $Y\in I \cap V_2$ with $X \subseteq Y$}\,\}$, and
  $S:=S_1 \cup S_2$. If $S_2$ is empty, then $I$ is a subset of $V_1$, and
  therefore $x(I)$ is at most one. On the other hand, if $S_1$ is empty,
  then $I$ is a subset of $V_2$. The graph induced by $V_2$ is
  $\twvec{p-3q/2\\q/2}$-regular, hence $x(I) \le p/q-2 < 1$.

  So we can assume that both $S_1$ and $S_2$ are non-empty. For a set
  $X \in S_2$, we define $\hat{x}(X):=
  \sum\limits_{Y \in I \cap V_2\;\text{s.t.}\; X \subseteq Y}x(Y)$. Now let
  $\O$ be the set of all $(p-q-1)!$ circular orders of the set
  $[p] \setminus [q]$. We say that a set $Z \subseteq [p]\setminus[q]$ is
  an \emph{arc} in $O \in \O$ if we can order the elements of $Z$ in such a
  way that they form a consecutive segment in $O$. For every $O \in \O$, we
  define the set $S^O$ as the subset of $S$ that contains $X \in S$ if and
  only if $X$ is an arc in $O$.

  Analogously, define $S_1^O$ as the family of sets $X \in S_1$ that are
  arcs in $O$, and $S_2^O$ as the family of sets $X \in S_2$ that are arcs
  in $O$. Observe that for every $X \in S_1$ there exist $q!(p-2q)!$
  choices of $O$ such that $X \in S_1^O$, and for every $X \in S_2$ there
  exist $(q/2)!(p-3q/2)!$ choices of $O$ such that $X \in S_2^O$.
  Consider the function $x': S \to [0,1]$ defined as follows:
  \vspace{-1mm}
  \begin{itemize}
  \item for $X \in S_1$, set $x'(X):=\dfrac{1}{p-q} =
    \dfrac{(p-q-1)!}{q!(p-2q)!} \cdot x(X)$; and
  \vspace{-1mm}
  \item for $X \in S_2$, set $x'(X):=\dfrac{2(p/q-2)\cdot|\{\,Y \in I:X
      \subseteq Y\,\}|}{(p-q)\cdot\twvec{q\\q/2}} =
    \dfrac{(p-q-1)!}{(q/2)!(p-3q/2)!} \cdot \hat{x}(X)$.
  \end{itemize}
  \vspace{-1mm}
  By a double counting argument,
  \begin{equation*}
    (p-q-1)!\cdot x(I) = \sum_{O\in \O}\sum_{X \in S^O} x'(X)\,.
  \end{equation*}
  Therefore, it is enough to show that for every $O \in \O$ the sum
  $\sum\limits_{X \in S^O} x'(X)$ is at most one. Let $x'(O)$ be this sum.

  Fix a circular order $O\in \O$. If $S_2^O$ is empty, then
  $x'(O) = \dfrac{|S_1^O|}{p-q} \le 1$. If $S_1^O$ is empty, then we show
  that $x'(O) \le p/q-2$. Indeed, consider the subgraph $H_O$ of $H$
  induced by $A \cup X$, where $A \in\twvec{[q]\\ q/2}$ and $X \in S_2$.
  Note that $|V(H_O)|\le(p-q)\twvec{q\\ q/2}$. By the definitions of $x'$
  and $\hat{x}$,
  \begin{equation*}
    x'(O) = \frac{\twvec{p-q\\ q/2}}{(p-q)} \cdot x(I \cap V(H_O)) =
    2(p/q-2) \cdot \frac{ |I \cap V(H_O)|}{(p-q)\twvec{q\\q/2}}\,.
  \end{equation*}
  Since the graph $H_O$ is $(p-2q+1)$-regular,
  $|I\cap V(H_O)|\le|V(H_O)|/2 \le (p-q)\twvec{q\\ q/2}/2$, and hence
  $x'(O) \le p/q-2$.

  Finally, consider the case that both $S_1^O$ and $S_2^O$ are non-empty.
  We claim that $|S^O| \le 3q/2$. We say that an arc $L$ in $O$ of size
  $q/2$ is \emph{forbidden} for $S_1^O$ if there exists a set in $S_1^O$
  that is disjoint from $L$. Let $s_1=|S_1^O|$ and $s_2=|S_2^O|$. Every arc
  in $O$ of size $q/2$ intersects at most $3q/2 - 1$ arcs in $O$ of size
  $q$, hence $s_1 \le 3q/2 - 1$. On the other hand, we show that at least
  $p-5q/2+s_1$ arcs in $O$ of size $q/2$ are forbidden for $S_1^O$, which
  means that $s_2 \le 3q/2-s_1$.

  Fix an arbitrary cyclic numbering of the elements of the set
  $[p]\setminus[q]$ with numbers $1,2,\dots,p-q$ such that any two
  consecutive elements in $O$ have consecutive numbers. Let~$K_\ell$, for
  $\ell \in [p-q]$, be the arc in $O$ of size $q$ that starts at the
  $\ell$-th element of $O$ and contains the next $q-1$ elements of $O$.
  Analogously, let $L_\ell$, for $\ell \in [p-q]$, be the arc of size~$q/2$
  that starts at the $\ell$-th element and contains the next $q/2-1$
  elements. For brevity, we also refer to $K_{p-q}$ as $K_0$, and to
  $L_{p-q}$ as $L_0$. If the sets in $S_1^O$ correspond to $s_1$
  consecutive arcs, i.e., for a fixed $\ell \in [p-q]$ the set $S_1^O$ is
  equal to $\{\,K_{\ell + j \mod p-q}:j=0,\dots,s_1-1\,\}$, then observe
  that exactly $p-5q/2+s_1$ arcs in $O$ are forbidden for $S_1^O$.

  Suppose now that the sets in $S_1^O$ do not correspond to $s_1$
  consecutive arcs. By symmetry, we may assume that $K_1 \in S_1^O$,
  $K_j \in S_1^O$ for some $j\in\{3,\ldots,p-q-1\}$, and
  $K_{j'}\notin S_1^O$ for every $j'=2,\ldots,j-1$. We will show that there
  exists a set~$T$ of~$s_1$ consecutive arcs in $O$ of size~$q$ such that
  the number of forbidden arcs in $O$ of size $q/2$ for~$S_1^O$ is at least
  the number of forbidden arcs for $T$. If the arc $L_{p-3q/2+2}$, i.e.,
  the arc that ends at the first element of~$O$, is disjoint from $K_j$,
  then every set that is disjoint from~$K_2$ is also disjoint from $K_1$ or
  $K_j$. Therefore, every arc in $O$ of size $q/2$ that is forbidden for
  $T':=(S_1^O\setminus \{K_j\}) \cup \{K_2\}$ is also forbidden for
  $S_1^O$.

  If the arc $L_{p-3q/2+2}$ intersects $K_j$, then since $K_{j'}$ is not in
  $S_1^O$ for all $j'=2,\ldots,j-1$, the arc $L_{j+q \mod p-q}$ is disjoint
  from $K_j$ and intersects every other set in $S_1^O$. Since it intersects
  also $K_2$, the number of forbidden arcs in $O$ of size $q/2$ for
  $T':=(S_1^O\setminus \{K_j\}) \cup \{K_2\}$ is at most the number of
  forbidden arcs for $S_1^O$. By repeating this procedure till the arcs
  in~$O$ in the set $T'$ are consecutive, we conclude that the number of
  forbidden arcs for $S_1^O$ is at least $p-5q/2+s_1$.

  Now if $s_1\ge q$, then
  \begin{equation*}
    x'(O) = \frac{s_1}{p-q} + \sum_{X \in S_2^O} x'(X) \le
    \frac{s_1 + (3q/2 - s_1) \cdot 2(p/q-2)}{p-q} \le 1 \,,
  \end{equation*}
  since the numerator of the last fraction is equal to
  $3p - 6q - s_1(2\cdot p/q-5)$, which is at most $p-q$.

  On the other hand, if $s_1<q$, then consider the partition of the set
  $\twvec{[q]\\ q/2}$ into $\twvec{q\\ q/2}/2$ unordered pairs $\{A,B\}$
  such that $A$ and $B$ are disjoint. Fix such a pair $\{A,B\}$. We claim
  that the number of tuples $(L,Z)$, where $L \in S_2^O$, $Z \in \{A,B\}$
  and $L \cup Z \in I$, is at most $q/2 + s_2$. Indeed, otherwise there
  would be at least $q/2 + 1$ arcs $L \in S_2^O$ such that both $L \cup A$
  and $L \cup B$ are in $I$. Since every arc in $O$ of size $q/2$
  intersects $q/2 - 1$ other arcs of size $q/2$, there exist two disjoint
  sets in $I$, which contradicts the fact that $I$ is an independent set.
  Therefore, it follows that
  \begin{equation*}
    \sum_{X \in S_2^O} x'(X) \le \frac{(q/2 + s_2)(p/q-2)}{(p-q)}
  \end{equation*}
  and
  \begin{equation*}
    x'(O) =\frac{s_1}{p-q} + \sum_{X \in S_2^O} x'(X) \le
    \frac{s_1 + (q/2+3q/2-s_1)(p/q-2)}{p-q} < 1\,.
  \end{equation*}
  The last inequality holds since the numerator of the last fraction is
  equal to $s_1(3-p/q)+2p-4q$, which is less than $p-q$.
\end{proof}

We are now ready to give a lower bound on $g(k,6)$ for $k\in[2.5,3)$.

\begin{theorem}
  \label{thm-lb-d6k25}
  Let $k\in[2.5,3)$ be a rational and $\eps$ a positive real such
  that $\eps < \dfrac1{k+\eps}\,$. There exist a graph $G$ with
  fractional chromatic number $k$, a subset $W$ of its vertex set at
  pairwise distance at least six and a fractional $(k+\eps)$-precoloring of
  $W$ that cannot be extended to a fractional $(k+\eps)$-coloring of $G$.
\end{theorem}

\begin{proof}
  Analogously to the proof of Theorem~\ref{thm-lb-d4}, let $\eps_0$ be the positive
  root of the equation $x=\dfrac1{k+x}$ and let $p'$ and $q$ be positive
  integers such that $q$ is even, $k+\eps \le p'/q < k+\eps_0$ and $kq$ is
  an integer. Next, set $p:=kq$, $\eps':=p'/q-k$ and $G:=U^n_{p,q,6}$,
  where $n=\twvec{p'\\ q}$. We will show the existence of a fractional
  $(k+\eps')$-precoloring of the special vertices of $G$ that cannot be
  extended to a fractional $(k+\eps')$-coloring of $G$, which implies the
  statement of the theorem by Proposition~\ref{prop-monotone}.

  Let $f:[p']\hookrightarrow 2^{[0,k+\eps')}$ be the function
  $f(i)=[(i-1)/q,i/q)$, for $i\in[p']$. Consider a precoloring of $G$ that
  assigns to the $n$ special vertices of the copies of $R_{p,q,3}^Y$,
  $Y \in\twvec{[p]\\ q}$, all the $n$ different sets $f(X)$, for
  $X\in\twvec{[p']\\ q}$. We assert that this fractional
  $(k+\eps')$-precoloring cannot be extended to a fractional coloring of
  the whole graph.

  Suppose, on the contrary, that there exists an extension of the
  precoloring given by $f$ to a fractional coloring
  $c:V(G)\to 2^{[0,k+\eps')}$. Let $H$ be the base of $G$ (recall that $H$
  is isomorphic to $K_{p/q}$), let $\I_H$ be the set of all independent
  sets in $H$, and for every $I\in \I_H$ let $\bar I$ be the complement of
  $I$ in $H$, i.e., $\bar I ={}V(H)\setminus I$.

  For every ray $R_i$ with its special vertex colored by $C_i$ and every
  independent set $I \in \I_H$, let $d_i(I)$ be the measure of the set of
  all points in $[0,k+\eps') \cap C_i$ assigned by $c$ to all vertices in
  $I$ and none in $\bar I$. In other words,
  \begin{equation*}
    d_i(I):=\Bigl\| \bigcap_{v\in I} (c(v) \cap C_i) \setminus c(\bar I)
    \Bigr\|\,.
  \end{equation*}
  Analogously, let $e_i(I)$ be the measure of points of
  $[0,k+\eps') \setminus C_i$ used, in the coloring $c$ restricted to $H$,
  exactly on the vertices of $I$, i.e.,
  \begin{equation*}
    e_i(I):=\Bigl\|\bigcap_{v\in I} (c(v) \setminus C_i)
    \setminus c(\bar I)\Bigr\|\,.
  \end{equation*}
  Finally, set $t_i(I):=d_i(I) + e_i(I)$. Observe that for every vertex
  $v\in V(H)$ the sum of $t_i(I)$ over all independent sets $I$ that
  contain $v$ is equal to one.

  Now let $N$ be the neighborhood in $H$ of the vertex $[q]$. (Where $[q]$
  is the vertex obtained by identifying the vertices $([q],3)$ from all
  rays $R_i$.) Recall that $|N|=\twvec{p-q\\ q}$. We assert that there
  exists a ray $R_i$ with special vertex $([q],0)$ for which
  \begin{equation}
    \label{ineq-d6k25-doublecount}
    \sum_{I \in \I_H} |N\cap I| \cdot d_i(I) \ge |N|\cdot\frac{1}{k+\eps'}
    = \twvec{p-q\\ q} \cdot\frac{q}{p'}\,.
  \end{equation}
  Indeed, let $C_1, C_2, \dots, C_n$ be the sets used in the precoloring of
  the vertex $([q],0)$ in the rays $R^{[q]}_{p,q,3}$. For simplicity, let
  $R_1, R_2, \dots, R_n$ be these rays and $d_1, e_1, \dots, d_n, e_n$ are
  the corresponding functions defined above. Since each point of
  $[0,k+\eps')$ is contained in exactly $\twvec{p'-1\\q'-1}$ sets $C_i$, it
  follows that $\sum\limits_{i=1}^n d_i(I)=\twvec{p'-1\\q'-1} \cdot t_i(I)$
  for every $I \in \I_H$. Next, by a double counting argument we have
  $\sum\limits_{I\in I_H} |I \cap N| \cdot t_i(I) =|N|=\twvec{p-q\\q}$.
  Therefore,
  \begin{equation*}
    \sum_{i=1}^n \sum_{I \in \I_H} |N\cap I| \cdot d_i(I) =
    \sum_{I \in \I_H} |N\cap I| \biggl(\sum_{i=1}^n d_i(I)\biggr) =
    \twvec{p'-1\\q-1}\cdot\twvec{p-q\\q}\,.
  \end{equation*}
  Since $n=\twvec{p'\\ q}$, there exists a ray $R_i$ with special vertex
  $([q],0)$ such that inequality~\eqref{ineq-d6k25-doublecount} holds. In
  the remainder of the proof, we fix $R_i$ to be such a ray and let $s$ be
  the special vertex in $R_i$.

  Let $\I_R$ be the set of all independent sets in the ray $R_i$ and let
  $V':=V(R_i) \setminus( V(H) \cup \{s\})$. Consider the following linear
  program $P$:
  \begin{alignat*}{2}
    \text{minimize: }    & \sum_{I \in \I_R} w(I); \\
    \text{subject to: } & \sum_{I \in \I_R,\;v\in I} w(I) \ge 1, &&
    \quad \forall v \in V'; \\
    & \sum_{\substack{I \in \I_R,\; s \in I \\[0.6mm] I \cap H = I_H}} w(I)
    \ge d_i(I_H), && \quad \forall I_H \in \I_H;\\
    & \sum_{\substack{I\in \I_R,\;s \notin I \\[0.6mm]I \cap H = I_H}} w(I)
    \ge e_i(I_H), && \quad \forall I_H \in \I_H;\\
    & w(I) \ge 0, && \quad \forall I \in \I_R.
  \end{alignat*}
  Observe that the fact that $c$ is a fractional $(k+\eps')$-coloring of
  $G$ implies that there exists a solution satisfying the conditions of $P$
  such that $\sum_{I \in \I_R} w(I) \le k+\eps'$. Now consider the dual
  program $P^*$ of $P$:
  \begin{alignat*}{2}
    \text{maximize: } &
    \sum_{v \in V'} y(v) + \sum_{I \in \I_H}\bigl[d_i(I) \cdot y^d(I)
    + e_i(I) \cdot y^e(I)\bigr];\\
    \text{subject to: } & y^d({I\cap H}) + \sum_{v \in I} y(v) \le 1,
    && \forall I \in \I_R\;\text{s.t.}\: s \in I; \\
    & y^e({I\cap H}) + \sum_{v \in I} y(v) \le 1,&&
    \forall I \in \I_R\;\text{s.t.}\; s \notin I;\\
    & y(v) \ge 0, & & \forall v \in V'; \\[1mm]
    & y^d(I) \geq0,\; y^e(I) \ge 0, && \forall I \in \I_H.
  \end{alignat*}
  We will show that there exists a feasible solution of $P^*$ such that the
  objective function of~$P^*$ is at least $k+q/p'=k+\dfrac1{k+\eps'}\,$.
  Therefore, $\eps' \ge{\dfrac1{k+\eps'}}\,$, which is a contradiction with
  the choice of $\eps'$.

  Let $x:\twvec{[p]\\ q} \to [0,1]$ be a fractional clique in $K_{p/q}$ of
  weight $p/q$ such that for every $X \in\twvec{[p]\setminus[q]\\ q}$ we
  have $x(X)=\twvec{p-q\\ q}^{-1}$. Proposition~\ref{prop-fclique} implies
  that such a clique exists (recall that $q$ is even). We now define an
  embedding $g$ of $K_{p/q}$ in the subgraph of~$R_i$ induced by $V'$. If
  $X\in\twvec{[p]\\ q}$ is disjoint from $[q]$, we set $g(X):=(X,1)$;
  otherwise we set $g(X):=(X,2)$. For every set $X \in\twvec{[p]\\ q}$ we
  set $y(g(X)):=x(X)$, and for every other vertex
  $v \in V' \setminus g\Bigl(\twvec{[p]\\ q}\Bigr)$ we set $y(v):=0$.
  Finally, we set $y^d(I):=\dfrac{|N \cap I|}{\twvec{p-q\\ q}}$ and
  $y^e(I):=0$ for every $I \in \I_H$.

  By the definition of $x$, it follows that
  $\sum\limits_{v \in V'} y(v) = k$. Next,
  inequality~\eqref{ineq-d6k25-doublecount} implies that
  $\sum\limits_{I \in \I_H} d_i(I) \cdot y^d(I) \ge q/p'$. Since
  $y$ forms a fractional clique in $R_i$, it remains to show that for
  $I \in \I_R$, where $s \in I$ and $I \cap N \neq \varnothing$, we have
  $y^d({I\cap H}) + \sum\limits_{v \in I} y(v)\le1$. We show that for every
  $I$, where $s \in I$ and $I \cap N \ne\varnothing$, there exists an
  independent set $I' \in \I_R$ that is disjoint from $N$ and that
  satisfies
  \begin{equation}
    \label{eq-d6k25-indepsum}
    y^d({I\cap H}) + \sum_{v \in I} y(v) = \sum_{v \in I'} y(v)\,.
  \end{equation}
  Since $y$ is a fractional clique in $R_i$ and $I'$ is an independent
  set, it follows that $\sum\limits_{v \in I'} y(v)\le1$. We construct $I'$
  in the following way:
  \vspace{-1mm}
  \begin{itemize}
  \item $I'$ is disjoint from $\Bigl\{\,(X,0): X\in\twvec{[p]\\q}\,\Bigr\}
    \cup \Bigl\{\,(X,3): X \in\twvec{[p]\\q}\,\Bigr\}$,
  \vspace{-1mm}
  \item $(X,1) \in I'$ if and only if $(X,3) \in I \cap N$ (observe that
    $(X,1) \notin I$), and
  \vspace{-1mm}
  \item $(X,2) \in I'$ if and only if $(X,2) \in I$.
  \end{itemize}
  \vspace{-1mm}
  By the choice of $y^d(H \cap I)$, and since
  $y((X,1))=\twvec{p-q\\ q}^{-1}$ for $X\in\twvec{[p]\setminus [q]\\ q}$,
  it follows that equation~\eqref{eq-d6k25-indepsum} holds. This completes
  the proof.
\end{proof}


\section{Distances congruent to one mod four}
\label{sect-upper-d1}

Analogously to Section~\ref{sect-upper-d0}, in this section we present
upper bounds on $g(k,d)$ for $d\equiv1\mod 4$ in the case that $k$ and $d$
satisfy $2\le k < 2 + \dfrac2{d-3}\,$. In Lemma~\ref{lemma-d0} we gave a
coloring strategy for universal graphs $U^n_{p,q,d}$ with $d$ even. In the
following lemma we adapt this strategy to odd values of $d$. Recall that
for odd~$d$, instead of identifying the bases of the rays in a universal
graph, we now connect them according to their labels. The main difference
in the new strategy is that to color the base of each ray of $U^n_{p,q,d}$,
we now use appropriately chosen subsets of the sets
$f_o(1), f_o(2), \dots, f_o(p)$, instead of using
$f_e(1), f_e(2), \dots, f_e(p)$.

\begin{lemma}
  \label{lemma-d1}
  Let $\eps$ be a positive real and $n$, $p$, $q$ and $d$ positive integers
  such that $d\ge5$, $d\equiv1\mod 4$ and $p/q\ge2$. If the conditions
  \begin{gather}
    \label{ineq-d1-og}
    2\le k < 2+\frac{1}{2d'-1}\quad \text{and}\\
    \label{ineq-d1-bound}
    \eps k \sum_{j=0}^{d'-1} (k-1)^{2j} \ge 1
  \end{gather}
  are satisfied, where $d'=\lfloor d/4\rfloor$ and $k=p/q$, then any
  fractional $(k+\eps)$-precoloring of the special vertices of
  $U_{p,q,d}^n$ can be extended to a fractional $(k+\eps)$-coloring of
  $U_{p,q,d}^n$.
\end{lemma}

\begin{proof}
  Again, we only need to consider $\eps$ that satisfy~\eqref{ineq-d1-bound}
  with equality, i.e., we can take
  \[\eps=\biggr(k\sum_{j=0}^{d'-1} (k-1)^{2j}\biggr)^{-1}.\]
  Note that this choice trivially gives
  \begin{equation}
    \label{ineq-d1-aux}
    \eps k \sum_{j=0}^{d'-2} (k-1)^{2j} \le 1.
  \end{equation}

  Considering the universal graph $U_{p,q,d}^n$, let $C_i$, for $i\in
  \Bigl[n\twvec{p\\q}\Bigr]$, be a~precoloring of the special vertices and
  let $f_o$ be a~mapping as described in
  Proposition~\ref{prop-pseudorandom}. In what follows, for each ray $R_i$,
  which is isomorphic to $R_{p,q,2d'}$, we find a fractional coloring $c_i$
  that satisfies the following: for every set $A \in\twvec{[p]\\q}$, each
  vertex $v=(A,2d')$ of the base of $R_i$ is colored by a subset of
  $f_o(A)$, and the special vertex of $R_i$ is colored by $C_i$. Since the
  universal graph $U_{p,q,d}^n$ is constructed by joining the vertices
  $(A,2d')$ and $(B,2d')$ from different rays for disjoint
  $A,B\in\twvec{[p]\\ q}$, the lemma follows from this claim.

  Fix a ray $R_i$ and let $s$ be the special vertex of $R_i$. For an
  integer $\ell \in [2d'-1]$, let $V_\ell$ be the set of vertices of $R_i$
  at distance $\ell$ from $s$, and let $V_{2d'}$ be the set of vertices of
  $R_i$ at distance at least $2d'$ from $s$. As in the proof of
  Lemma~\ref{lemma-d0}, the vertices of the base of $R_i$ form a subset of
  $V_{2d'}$ and the set $V_\ell$ forms an independent set in $R_i$ for
  $\ell \in [2d'-1]$.

  Analogously to the proof of Lemma~\ref{lemma-d0}, we construct functions
  $f_x:[p] \hookrightarrow 2^{[0,k+\eps)}$,
  $g_y:[p] \hookrightarrow 2^{[0,k+\eps)}$ and
  $h_z:[p] \hookrightarrow 2^{[0,k+\eps)}$, for $x \in [2d']$, $y\in [d']$
  and $z\in [d'-1]$ as follows. For $a \in [p]$ and
  $j = d'-1, d'-2, \dots, 1$, we sequentially define
  \vspace{-1mm}
  \begin{itemize}
  \item $g_{d'}(a)$ as an arbitrary subset of $f_o(a)\setminus C_i$ of
    measure $\dfrac{\eps}{p}\,$,
  \vspace{-1mm}
  \item $g_{j}(a)$ as an arbitrary subset of $(f_o(a)\setminus C_i)
    \setminus \bigcup\limits_{j'=j+1}^{d'}g_{j'}(a)$\\[-2mm]
    \hspace*{\fill} of measure $\dfrac{\eps k}{p}(k-1)^{2(d'-j)-1}$,
  \vspace{-1mm}
  \item $h_{j}(a)$ as an arbitrary subset of $(f_o(a)\cap C_i) \setminus
    \bigcup\limits_{j'=j+1}^{d'-1}h_{j'}(a)$\\[-2mm]
    \hspace*{\fill} of measure $\dfrac{\eps k}{p}(k-1)^{2(d'-j)-2}$,
  \end{itemize}
  \vspace{-1mm}
  and then:
  \vspace{-1mm}
  \begin{itemize}
  \item $f_{2d'}(a):=f_o(a)\setminus g_{d'}(a)$,
  \vspace{-1mm}
  \item $f_{2j+1}(a):=f_{2j+2}(a)\setminus h_j(a)$, and
  \vspace{-1mm}
  \item $f_{2j}(a):=f_{2j+1}(a)\setminus g_j(a)$.
  \end{itemize}
  \vspace{-1mm}
  Finally, we set $f_1(a):=f_2(a)\setminus C_i$ for every $a\in [p]$. Since
  the measure of~$f_o(a)$ is $(k+\eps)/p$ and the measure of
  $f_o(a)\cap C_i$ is $1/p$, these functions exist if and only if condition
  \eqref{ineq-d1-aux} is satisfied. The described construction of the
  functions is sketched in Figure~\ref{fig-f-d1}.

  \begin{figure}
    \begin{center}
      \epsfxsize=135mm
      \epsfbox{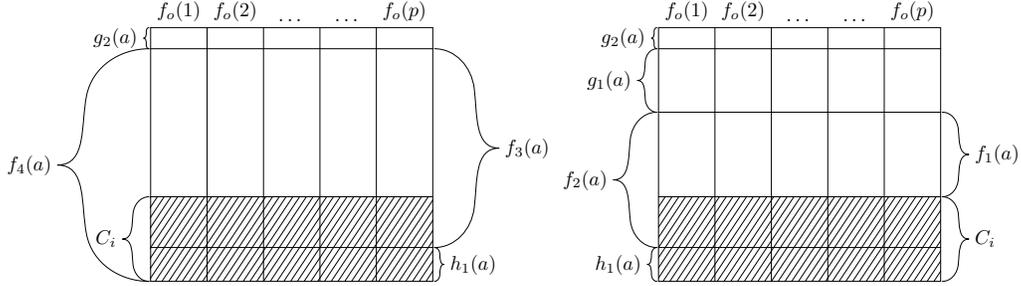}
    \end{center}
    \caption{The construction of a fractional coloring in
      Lemma~\ref{lemma-d1} for $d=9$.}
    \label{fig-f-d1}
  \end{figure}

  Let $\ell \in [2d']$ and $v=(A,\ell') \in V_\ell$. Recall that
  $\ell'\le \ell$. If $\ell$ is even, we set
  \begin{equation*}
    c_i(v):=f_\ell(A) \cup \bigcup_{j=\ell/2}^{d'-1} h_j([p])\,;
  \end{equation*}
  and if $\ell$ is odd, we set
  \begin{equation*}
    c_i(v):=f_\ell(A) \cup \bigcup_{j={(\ell+1)/2}}^{d'} g_j([p])\,.
  \end{equation*}
  Finally, we set $c_i(s):=C_i$.

  As in the proof of Lemma~\ref{lemma-d0}, we claim that $\|c_i(v)\|\ge1$
  for every vertex $v\in V(R_i)$. First, it follows from the definition
  that $\|c_i(s)\|=1$. Next, for a fixed $\ell \in [2d']$, the color sets
  of any two vertices $u$ and $v$ from $V_\ell$ have the same measure. Let
  $m_\ell$ be the measure of vertices in $V_\ell$. Then
  $m_{2d'}=\dfrac{k+\eps}{k} - \dfrac{\eps}{k} = 1$. Next, for $\ell \in
  \{2,3,\dots,2d'-1\}$ we have $m_\ell=1$, by analogous calculations as in
  the proof of Lemma~\ref{lemma-d0}. Finally, for $m_1$ we obtain
  \begin{equation*}
    m_1=1 - \frac{1}{k} - (k-1) \cdot \eps
    k\sum_{j=1}^{d'-1}(k-1)^{2j-1} + \eps = 1 - \frac{1}{k}
    +\eps\sum_{j=0}^{d'-1}(k-1)^{2j}\,,
  \end{equation*}
  which is at least one by \eqref{ineq-d1-bound}.

  An analysis analogous to that presented in the proof of
  Lemma~\ref{lemma-d0} yields that $c_i$ assigns disjoint sets to any two
  adjacent vertices in $R_i$. Therefore, the coloring $c_i$ is a fractional
  coloring of the ray $R_i$ with the required properties.
\end{proof}

As in Section~\ref{sect-upper-d0}, applying Proposition~\ref{prop-homo}
yields the following theorem.

\begin{theorem}
  \label{thm-d1}
  Let $d\ge5$ be an integer such that $d\equiv1\mod4$, $k$ a rational and
  $\eps$ a positive real such that conditions \eqref{ineq-d1-og} and
  \eqref{ineq-d1-bound} are satisfied, where $d'=\lfloor d/4\rfloor$. If
  $G$ is a fractionally $k$-colorable graph and $W$ is a subset of its
  vertex set with pairwise distance at least $d$, then any fractional
  $(k+\eps)$-precoloring of $W$ can be extended to a fractional
  $(k+\eps)$-coloring of $G$.
\end{theorem}

Note that for $d=5$, the theorem shows that $g(k,5)\le1/k$ for $k\in[2,3)$.


\section{Distances congruent to three mod four}
\label{sect-upper-d3}

As in the previous sections, we start with showing upper bounds on $g(k,d)$
for $d\equiv3 \mod 4$ such that $k$ and $d$ satisfy the condition
$2\le k < 2 + \dfrac2{d-3}\,$. Observe that the parity of the length of a
ray in $U^n_{p,q,d}$ for $d\equiv 3\mod 4$ is the same as for
$d\equiv 2\mod 4$, and that it is different to the one for
$d\equiv 1\mod 4$. (Hence it is also different to the one for
$d\equiv 0\mod 4$.) Therefore, for this choice of values of $k$ and $d$, we
modify the coloring strategy used in Lemma~\ref{lemma-d1} in a similar way
to how we modified the strategy from Lemma~\ref{lemma-d0} to prove
Lemma~\ref{lemma-d2}.

\begin{lemma}
  \label{lemma-d3}
  Let $\eps$ be a positive real and $n$, $p$, $q$ and $d$ positive integers
  such that $d\equiv3 \mod 4$ and $p/q\ge2$. If the conditions
  \begin{gather}
    \label{ineq-d3-og}
    2\le k < 2+\dfrac{1}{2d'} \quad\text{and}\\
    \label{ineq-d3-bound}
    \eps + \eps k \sum_{j=0}^{d'-1} (k-1)^{2j+1} \ge 1
  \end{gather}
  are satisfied, where $d'=\lfloor d/4\rfloor$ and $k=p/q$, then any
  fractional $(k+\eps)$-precoloring of the special vertices of
  $U_{p,q,d}^n$ can be extended to a fractional $(k+\eps)$-coloring of
  $U_{p,q,d}^n$.
\end{lemma}

\begin{proof}
  For the fourth time, we can limit ourselves to $\eps$ that give equality
  in~\eqref{ineq-d3-bound}:
  \[\eps=\biggl(1+ k \sum_{j=0}^{d'-1} (k-1)^{2j+1}\biggr)^{-1}.\]
  For later in the proof we observe that these $\eps$ trivially satisfy
  \begin{equation}
    \label{ineq-d3-aux}
    \eps + \eps k \sum_{j=0}^{d'-2} (k-1)^{2j+1} \le 1.
  \end{equation}

  For the universal graph $U_{p,q,d}^n$, let $C_i$, for $i\in
  \Bigl[n\twvec{p\\q}\Bigr]$, be a~precoloring of the special vertices and
  $f_o$ be a~mapping as described in Proposition~\ref{prop-pseudorandom}.
  Analogously to the proof of Lemma~\ref{lemma-d1}, for each ray $R_i$ we
  find a fractional coloring $c_i$ that satisfies the following: every vertex $v=(A,2d'+1)$ of
  the base of $R_i$ is colored by a subset of $f_o(A)$,
  where $A \in\twvec{[p]\\q}$,  and the special vertex of $R_i$ is colored by $C_i$.

  Fix a ray $R_i$ and let $s$ be the special vertex of $R_i$. For an
  integer $\ell \in [2d']$, let $V_\ell\subseteq V(R_i)$ be the set of
  vertices at distance $\ell$ from $s$, and let $V_{2d'+1}$ be the set of
  vertices of $R_i$ at distance at least $2d'+1$ from $s$. Similarly as in
  the proof of Lemma~\ref{lemma-d1}, the vertices of the base of $R_i$ form
  a subset of $V_{2d'+1}$ and $V_\ell$ forms an independent set in $R_i$
  for $\ell \in [2d']$.

  We now construct functions $f_x:[p] \hookrightarrow 2^{[0,k+\eps)}$,
  $g_y:[p] \hookrightarrow 2^{[0,k+\eps)}$ and
  $h_z:[p] \hookrightarrow 2^{[0,k+\eps)}$, for $x \in [2d'+1]$,
  $y\in [d']$ and $z\in [d']$ as follows. For $a \in [p]$ and
  $j=d'-1,d'-2,\dots,1$, we sequentially define:
 \vspace{-1mm}
  \begin{itemize}
  \item $h_{d'}(a)$ as an arbitrary subset of $f_o(a)\cap C_i$ of measure
    $\dfrac\eps{p}\,$, and
  \vspace{-1mm}
  \item $h_{j}(a)$ as an arbitrary subset of $(f_o(a)\cap C_i) \setminus
    \bigcup\limits_{j'=j+1}^{d'}h_{j'}(a)$\\[-2mm]
    \hspace*{\fill} of measure $\dfrac{\eps k}{p}(k-1)^{2(d'-j)-1}$.
  \end{itemize}
  \vspace{-1mm}
  Next, we sequentially define for $a \in [p]$ and $m = d',d'-1,\dots,1$
  \vspace{-1mm}
  \begin{itemize}
  \item $g_{m}(a)$ as an arbitrary subset of $(f_o(a)\setminus C_i)
    \setminus \bigcup\limits_{m'=m+1}^{d'}g_{m'}(a)$\\[-2mm]
    \hspace*{\fill} of measure $\dfrac{\eps k}{p}(k-1)^{2(d'-m)}$,
  \vspace{-1mm}
  \item $f_{2d'+1}(a):=f_o(a)\setminus h_{d'}(a)$,
  \vspace{-1mm}
  \item $f_{2m+1}(a):=f_{2m+2}(a)\setminus h_m(a)$ for $m < d'$, and
  \vspace{-1mm}
  \item $f_{2m}(a):=f_{2m+1}(a)\setminus g_m(a)$.
  \end{itemize}
  \vspace{-1mm}
  Finally, we define $f_1(a):=f_2(a)\setminus C_i$ for every $a\in [p]$.
  Similarly as in the proof of Lemma~\ref{lemma-d1}, these functions exist
  if and only if condition \eqref{ineq-d3-aux} is satisfied. The described
  construction of the functions is sketched in Figure~\ref{fig-f-d3}.

  \begin{figure}
    \begin{center}
      \epsfxsize=135mm
      \epsfbox{ext-d3.3}
    \end{center}
    \caption{The construction of a fractional coloring in
      Lemma~\ref{lemma-d3} for $d=7$.}
    \label{fig-f-d3}
  \end{figure}

  Let $\ell \in [2d']$ and $v=(A,\ell') \in V_\ell$. If $\ell$ is even, we
  set
  \begin{equation*}
    c_i(v):=f_\ell(A) \cup \bigcup_{j=\ell/2}^{d'} h_j([p])\,;
  \end{equation*}
  and if $\ell$ is odd, we set
  \begin{equation*}
    c_i(v):=f_\ell(A) \cup \Bigcup_{j={(\ell+1)/2}}^{d'} g_j([p])\,.
  \end{equation*}

  Also, set $c_i(s):=C_i$. An analysis analogous to that presented in the
  proof of Lemma~\ref{lemma-d0} yields that $c_i$ is a fractional coloring
  of the ray $R_i$ with the required properties.
\end{proof}

Lemma~\ref{lemma-d3} and Proposition~\ref{prop-homo} together provide the
following theorem.

\begin{theorem}
  \label{thm-d3}
  Let $d$ be a positive integer such that $d\equiv3\mod4$, $k$ a rational
  and $\eps$ a positive real such that conditions \eqref{ineq-d3-og} and
  \eqref{ineq-d3-bound} are satisfied, where $d'=\lfloor d/4\rfloor$. If
  $G$ is a fractionally $k$-colorable graph and $W$ is a subset of its
  vertex set with pairwise distance at least $d$, then any fractional
  $(k+\eps)$-precoloring of $W$ can be extended to a fractional
  $(k+\eps)$-coloring of $G$.
\end{theorem}

For $d=7$, the theorem means that $g(k,7)\le\dfrac1{k^2-k+1}$ for
$k\in[2,2.5)$. We close this section by showing an upper bound on $g(k,7)$
for $k\in [2.5,3)$.

\begin{theorem}
  \label{thm-d7k25}
  Let $k$ be a positive rational less than 3 and $\eps$ a positive real
  such that $\eps \ge \dfrac1{k+1}\,$. If $G$ is a fractionally
  $k$-colorable graph and $W$ is a subset of its vertex set with pairwise
  distance at least seven, then any fractional $(k+\eps)$-precoloring of
  $W$ can be extended to a fractional $(k+\eps)$-coloring of $G$.
\end{theorem}

\begin{proof}
  By Proposition~\ref{prop-homo}, it is enough to consider the universal
  graphs $U_{p,q,6}^n$, where $p/q=k$ and $n\in \N$, and an arbitrary
  precoloring of its special vertices. Furthermore, we may assume
  $\eps \le 1$, since $g(k,3)=1$ for every $k\ge2$ by
  Theorem~\ref{thm-ext1}.

  As in the proofs of Lemmas~\ref{lemma-d1} and \ref{lemma-d3}, let $C_i$,
  for $i\in \Bigl[n\twvec{p\\q}\Bigr]$, be a precoloring of the special
  vertices and let $f_o$ be a~mapping as described in
  Proposition~\ref{prop-pseudorandom}. For each ray $R_i$ we find a
  fractional coloring $c_i$ that satisfies the following: for every set
  $A \in\twvec{[p]\\q}$, each vertex $v=(A,3)$ of the base of $R_i$ is
  colored by a subset of $f_o(A)$, and the special vertex of $R_i$ is
  colored by $C_i$.

  Fix a ray $R_i$ and let $s$ be the special vertex of $R_i$. By symmetry,
  it is enough to consider the case where $R_i$ is a copy of
  $R_{p,q,3}^{[q]}$. We construct functions
  $g_2:[p]\hookrightarrow 2^{[0,k+\eps)}$,
  $g_1:[q]\hookrightarrow 2^{[0,k+\eps)}$ and
  $h:[q]\hookrightarrow 2^{[0,k+\eps)}$ as follows. For $j\in[q]$ and
  $j'\in[p]\setminus [q]$ we define:
  \vspace{-1mm}
  \begin{itemize}
  \item $g_2(j)$ as an arbitrary subset of $f_o(j)\setminus C_i$ of measure
    $\dfrac\eps{p}\,$,
  \vspace{-1mm}
  \item $g_2(j')$ as an arbitrary subset of $f_o(j')\cap C_i$ of measure
    $\dfrac\eps{p}\,$, and
  \vspace{-1mm}
  \item $g_1(j)$ as an arbitrary subset of $(f_o(j)\setminus C_i) \setminus
    g_2(j)$ of measure $\dfrac\eps{p}(k-1)$.
  \end{itemize}
  \vspace{-1mm}
  Note that these functions exist if and only if $\eps \le 1$. Next, we
  define sets $h(1), h(2), \dots, h(q)$ as an arbitrary equipartition of
  $g_2([p]\setminus[q])$ into~$q$ parts of measure $\dfrac\eps{p}(k-1)$.
  The described construction of the functions is sketched in
  Figure~\ref{fig-f-d7k25}.

  \begin{figure}
    \begin{center}
      \epsfxsize=120mm
      \epsfbox{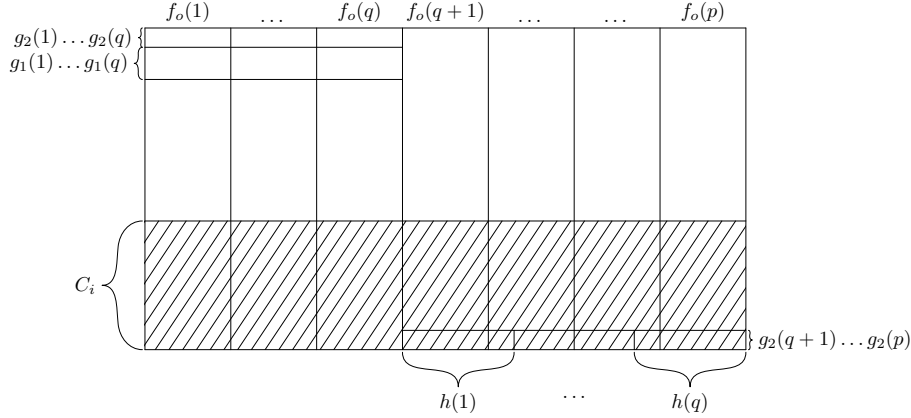}
    \end{center}
    \caption{The construction of a fractional coloring in
      Theorem~\ref{thm-d7k25}.}
    \label{fig-f-d7k25}
  \end{figure}

  Recall that the neighborhood of $s$ in $R_i$ forms an independent set.
  Since we assume that $s=([q],0)$, for every neighbor $(A,\ell')$ of $s$
  we have $A\cap[q]=\varnothing$ and $\ell'=1$. We now construct a
  fractional coloring of $R_i$. Let $v=(A,\ell')$ be a vertex of $R_i$ and
  let $\ell$ be the distance of $v$ and $s$ in $R_i$. We define $c_i(v)$ in
  the following way:
  \vspace{-1mm}
  \begin{itemize}
  \item if $\ell\ge3$, then $c_i(v):=f_o(A)\setminus g_2(A)$,
  \vspace{-1mm}
  \item if $\ell=2$, then $c_i(v):=\bigl(( f_o(A) \setminus g_2(A) )
    \setminus g_1(A\cap[q]) \bigr) \cup h(A \cap [q])$,
  \vspace{-1mm}
  \item if $\ell=1$, then $c_i(v):=( f_o(A) \setminus C_i) \cup g_1([q])
    \cup g_2([q])$, and
  \vspace{-1mm}
  \item $c_i(s):=C_i$.
  \end{itemize}
  \vspace{-1mm}
  An analysis analogous to that presented in the proof of
  Lemma~\ref{lemma-d0} yields that we assigned disjoint sets to any two
  neighbors in $R_i$, and that any vertex at distance at least two from $s$
  got a set of measure one. Furthermore, for every
  $A\in {[p]\setminus [q] \choose q}$ the set $f_o(A)\setminus C_i$ is
  disjoint from both $g_1([q])$ and $g_2([q])$, and it has measure
  $1-(1-\eps)/k$. Since $g_1([q]) \cup g_2([q])$ has measure $\eps$ and for
  $\eps \ge 1/(k+1)$ we have $\eps \ge (1-\eps)/k$, it follows that $c_i$
  is a fractional coloring with the required properties.
\end{proof}


\section{Open problems}
\label{sect-open}

Determining further values of $g(k,d)$ seems to require additional
knowledge on the structure of independent sets and fractional colorings in
Kneser graphs. We believe that our upper bounds presented in
Theorems~\ref{thm-d0}, \ref{thm-d2}, \ref{thm-d1}, \ref{thm-d3},
and~\ref{thm-d7k25} are tight. In particular, for distances $d=5,6$ and
$7$, we conjecture the following.

\begin{conjecture}
\label{conj-d5}
  For $k\in[2,3)$ we have $g(k,5)=\dfrac1{k}\,$.
\end{conjecture}

\begin{conjecture}
\label{conj-d6}
  For $k\in[2,2.5)$ we have
  $g(k,6)=\frac12\bigl(\sqrt{k^2+4/(k-1)}-k\bigr)$.
\end{conjecture}

\begin{conjecture}
\label{conj-d7}
  For $k\in[2,2.5)$ we have $g(k,7)=\dfrac1{k^2-k+1}\,$, while for
  $k\in[2.5,3)$ we have $g(k,7)=\dfrac1{k+1}\,$.
\end{conjecture}

Let us give some additional support for
Conjectures~\ref{conj-d5}--\ref{conj-d7}, provided by numerical
computations. Fix $\eps > 0$, integers $n$, $p$, $q$ and $d$ such that
$k=p/q\in[2,3)$, and a fractional precoloring $C_i\subseteq[0,k+\eps)$, for
$i\in\Bigl[n\twvec{p\\q}\Bigr]$, of the special vertices of the universal
graph~$U^n_{p,q,d}$. Denote by $R_i$ the ray $U^n_{p,q,d}$ that has the
special vertex precolored with $C_i$. Finally, fix the coloring of the
bases of the rays of $U^n_{p,q,d}$ as given in the following paragraph.
(Recall that every base is isomorphic to $K_{p/q}$, and that for even
values of $d$, the bases of all the rays are actually the same.)

Let $f_e$ and $f_o$ be the functions defined in
Proposition~\ref{prop-pseudorandom} for the precoloring $C_i$. If $d$ is
even, then the base of each ray of $U^n_{p,q,d}$ is colored using the sets
$f_e(1),f_e(2),\dots,f_e(p)$. If $d\equiv 1 \mod 4$, then for each
$i\in\Bigl[n\twvec{p\\q}\Bigr]$, the base of $R_i$ is colored using
(arbitrarily chosen) sets
$f^i_o(1) \subset f_o(1), f^i_o(2) \subset f_o(2), \dots, f^i_o(p) \subset
f_o(p)$ of measure $1/q$ satisfying $\|f^i_o(j)\cap C_i\|=1/p$ for every
$j\in[p]$. Finally, if $d\equiv 3\mod 4$, then the base of the ray $R_i$ is
colored by sets
$f^i_o(1) \subset f_o(1), f^i_o(2) \subset f_o(2), \dots, f^i_o(p) \subset
f_o(p)$ of measure $1/q$ satisfying $\|f^i_o(j)\cap C_i\|=(1-\eps)/p$ for
every $j\in[p]$.

Observe that the precoloring of the special vertices and the coloring of
the bases of $U^n_{p,q,d}$ can be extended to a fractional
$(p/q+\eps)$-coloring if and only if we can extend this precoloring inside
each ray $R_i$ separately. Furthermore, the question if we can extend this
precoloring inside $R_i$ can be formulated as a linear program similar to
the program~$P$ defined in Theorem~\ref{thm-d6k25}. With the help of the
QSopt Linear Programming Solver~\cite{bib-qsopt}, we have checked the
minimum possible values of $\eps$ for various choices of integers $p$, $q$
and $d$. All the numerical values matched the values we conjectured above;
see Table~\ref{tab-linearprogram}.

Note that the assumption on the coloring of the bases of the rays we made
is satisfied in both the optimal extension for $d=4$ and $k\in [2,3)$ from
Theorem~\ref{thmain-d4}, and the optimal extension for $d=6$ and
$k \in [2.5,3)$ from Theorem~\ref{thmain-d6k25}. This is also the case for
the optimal extensions for $d=3$ and $k\in [2,\infty)$, and 
$d \ge 4$ and $k\in\{2\}\cup[3,\infty)$, which were constructed in~\cite{bib-extenze}.

\begin{table}
\begin{center}
\begin{tabular}{ccccccc}
\toprule
$k$ & $p$ & $q$ & $\eps$ for $d=5$ & $\eps$ for $d=6$ & $\eps$ for $d=7$ & $\eps$ for $d=8$\\
\midrule
$2.07692$ & $27$ & $13$ & $0.48148$ & $0.37822$ & --        & --\\
\midrule
$2.08333$ & $25$ & $12$ & $0.48$    & $0.37542$ & --        & --\\
\midrule
$2.09091$ & $23$ & $11$ & $0.47826$ & $0.37216$ & --        & --\\
\midrule
$2.1$     & $21$ & $10$ & $0.47619$ & $0.36831$ & --        & $0.24$\\
\midrule
$2.11111$ & $19$ & $9$  & $0.47368$ & $0.36367$ & $0.29889$ & --\\
\midrule
$2.125$   & $17$ & $8$  & $0.47059$ & $0.358$   & $0.29493$ & --\\
\midrule
$2.14286$ & $15$ & $7$  & $0.46667$ & $0.35088$ & $0.28994$ & $0.22427$\\
\midrule
$2.16667$ & $13$ & $6$  & $0.46154$ & $0.34171$ & $0.28346$ & $0.21616$\\
\midrule
$2.2$     & $11$ & $5$  & $0.45454$ & $0.32945$ & $0.27472$ & $0.20538$\\
\midrule
$2.25$    & $9$ & $4$   & $0.44444$ & $0.31223$ & $0.26229$ & $0.19035$\\
\midrule
$2.28571$ & $16$ & $7$  & $0.4375$  & $0.30071$ & --        & --\\
\midrule
$2.33333$ & $7$ & $3$   & $0.42857$ & $0.2863$  & $0.24324$ & $0.20657$\\
\midrule
$2.33333$ & $14$ & $6$  & $0.42857$ & $0.2863$  & --        & --\\
\midrule
$2.4$     & $12$ & $5$  & $0.41667$ & $0.26775$ & $0.22936$ & --\\
\midrule
$2.5$     & $5$ & $2$  & $0.4$      & --        & $0.28571$ & $0.23892$\\
\midrule
$2.5$     & $10$ & $4$ & $0.4$      & --        & $0.28571$ & $0.23892$\\
\midrule
$2.66667$ & $8$  & $3$ & $0.375$    & --        & $0.27273$ & $0.23274$\\
\midrule
$2.75$    & $11$ & $4$ & $0.36364$  & --        & $0.26667$ & --\\
\bottomrule
\end{tabular}
\end{center}
\caption{Minimum possible values of $\eps$ for several choices of $p$,
  $q$ and $d$, obtained by~numerical computations.}
\label{tab-linearprogram}
\end{table}

Finally, we also believe that as $d$ gets larger, the function $g(k,d)$ is
discontinuous for more values of $k$. In particular, let $m(p,q,d)$ be the
maximum integer $i$ such that $i \le \lfloor d/2\rfloor -1$ and the $i$-th
neighborhood of the special vertex of $R_{p,q,\lfloor d/2\rfloor}$ form an
independent set. Observe that for any integers $d$, $p$, $q$, $p'$ and $q'$
such that $p/q = p'/q'$, the values of $m(p,q,d)$ and $m(p',q',d)$ are the
same. We expect that the discontinuous points of $g(k,d)$ exactly
correspond to those values of $k=p/q$, where $m(p,q,d)$ changes from a
value $\ell$ to $\ell+1$. Therefore, we pose the following conjecture.

\begin{conjecture}
  For a fixed integer $d\ge4$, the function $g(k,d)$ is discontinuous at
  $k\in[2,\infty)$ if and only if $k=2+1/m$ with
  $m\in\{1,2,\ldots,\lfloor d/2\rfloor-1\}$.
\end{conjecture}


\end{document}